\RequirePackage{fix-cm}
\documentclass[smallextended]{svjour3}       
\smartqed  
\usepackage{cite}
\usepackage{amsmath}
\usepackage{amsfonts}
\usepackage{amssymb}
\usepackage{bm}
\usepackage{subfigure}
\usepackage{xspace}
\usepackage{afterpage}
\usepackage{color}
\usepackage{pstricks}
\usepackage[top=1in, bottom=1in, left=1in, right=2in]{geometry}
\usepackage[textwidth=1.5in,textsize=small]{todonotes}
\usepackage{hyperref} \hypersetup{debug=true, backref=true,
letterpaper=true, colorlinks=true, bookmarks, linkcolor=blue, citecolor=blue}

\newcommand{\Tr}{\ensuremath{^{\mr{T}}}}

\newcommand{\mr}[1]{\ensuremath{\mathrm{#1}}}
\newcommand{\fnc}[1]{\ensuremath{\mathcal{#1}}}
\newcommand{\mat}[1]{\ensuremath{\mathsf{#1}}}
\newcommand{\M}[0]{\mat{H}}
\newcommand{\Jac}[0]{\mat{J}}
\newcommand{\JacL}[0]{\Jac_{L}}
\newcommand{\JacR}[0]{\Jac_{R}}
\newcommand{\ML}[0]{\mat{H}_{L}}
\newcommand{\MR}[0]{\mat{H}_{R}}
\newcommand{\Dx}[0]{\mat{D}_{\xi}}
\newcommand{\Dy}[0]{\mat{D}_{\eta}}

\newcommand{\Sx}[0]{\mat{S}_{\xi}}
\newcommand{\Sy}[0]{\mat{S}_{\eta}}

\newcommand{\Qx}[0]{\mat{Q}_{\xi}}
\newcommand{\Qy}[0]{\mat{Q}_{\eta}}
\newcommand{\DxR}[0]{\mat{D}_{\xi R}}
\newcommand{\DyR}[0]{\mat{D}_{\eta R}}
\newcommand{\DxL}[0]{\mat{D}_{\xi L}}
\newcommand{\DyL}[0]{\mat{D}_{\eta L}}
\newcommand{\QxL}[0]{\mat{Q}_{\xi L}}
\newcommand{\QyL}[0]{\mat{Q}_{\eta L}}
\newcommand{\QxR}[0]{\mat{Q}_{\xi R}}
\newcommand{\QyR}[0]{\mat{Q}_{\eta R}}

\newcommand{\ExL}[0]{\mat{E}_{\xi L}}
\newcommand{\ExLi}[1]{\mat{E}_{\xi L,#1}}

\newcommand{\ExRi}[1]{\mat{E}_{\xi R,#1}}

\newcommand{\EyL}[0]{\mat{E}_{\eta L}}

\newcommand{\ExR}[0]{\mat{E}_{\xi R}}
\newcommand{\EyR}[0]{\mat{E}_{\eta R}}

\newcommand{\oneL}[0]{\bm{1}_{L}}
\newcommand{\oneR}[0]{\bm{1}_{R}}
\newcommand{\oneG}[0]{\bm{1}_{\Gamhat}}

\newcommand{\Ex}[0]{\mat{E}_{\xi}}
\newcommand{\Ey}[0]{\mat{E}_{\eta}}

\newcommand{\B}[0]{\mat{B}}
\newcommand{\R}[0]{\mat{R}}
\newcommand{\BL}[0]{\B_{L}}

\newcommand{\RL}[0]{\R_{L}}
\newcommand{\BR}[0]{\B_{R}}
\newcommand{\RR}[0]{\R_{R}}
\newcommand{\Blam}[0]{\B_{\lambda}}
\newcommand{\V}[0]{\mat{V}}

\newcommand{\Vx}[0]{\mat{V}_{\xi}}

\newcommand{\nmin}[1]{N^{*}_{#1}}
\newcommand{\uL}[0]{\bm{u}_{L}}
\newcommand{\uR}[0]{\bm{u}_{R}}

\newcommand{\ignore}[1]{}
\newcommand{\etal}[0]{{\em et~al.\@}\xspace}
\newcommand{\eg}[0]{{e.g.\@}\xspace}

\newcommand{\ie}[0]{{i.e.\@}\xspace}

\newcommand{\fpk}{\fnc{P}_{k}}
\newcommand{\fpm}{\fnc{P}_{m}}

\newcommand{\pk}{\bm{p}_{k}}
\newcommand{\pM}{\bm{p}_{m}}
\newcommand{\pkL}{\bm{p}_{k,L}}

\newcommand{\dxfpm}{\frac{\partial\fnc{P}_{m}}{\partial \xi}}

\newcommand{\dxpk}{\bm{p}_{k}'}

\newcommand{\veclambda}[0]{\ensuremath{\boldsymbol{\lambda}}}
\newcommand{\vecnabla}[0]{\ensuremath{\nabla}}
\newcommand{\vecnrm}[0]{\ensuremath{\bm{n}}}

\newcommand{\MLL}[0]{\mat{M}^{\lambda_{\xi}}_{LL}}
\newcommand{\MLR}[0]{\mat{M}^{\lambda_{\xi}}_{LR}}
\newcommand{\MRR}[0]{\mat{M}^{\lambda_{\xi}}_{RR}}
\newcommand{\MRL}[0]{\mat{M}^{\lambda_{\xi}}_{RL}}

\newcommand{\SBPGamma}[0]{{SBP-$\Gamma$\@}\xspace}
\newcommand{\SBPOmega}[0]{{SBP-$\Omega$\@}\xspace}

\DeclareMathOperator{\mydiag}{diag}

\newtheorem{thrm}{Theorem}
\newtheorem{assume}{Assumption}
\newtheorem{cond}{Condition}

\newtheorem{innercustomcond}{Condition}
\newenvironment{customcond}[1]
  {\renewcommand\theinnercustomcond{#1}\innercustomcond}
  {\endinnercustomcond}

\newcommand{\lamxx}[0]{\lambda_{x}}
\newcommand{\lamyy}[0]{\lambda_{y}}

\newcommand{\lamx}[0]{\lambda_{\xi}}
\newcommand{\lamy}[0]{\lambda_{\eta}}
\newcommand{\Lamx}[0]{\Lambda_{\xi}}
\newcommand{\Lamy}[0]{\Lambda_{\eta}}

\newcommand{\nx}[0]{n_{\xi}}
\newcommand{\ny}[0]{n_{\eta}}
\newcommand{\nxL}[0]{n_{\xi L}}
\newcommand{\nyL}[0]{n_{\eta L}}
\newcommand{\nxR}[0]{n_{\xi R}}
\newcommand{\nyR}[0]{n_{\eta R}}
\newcommand{\nxx}[0]{n_{x}}
\newcommand{\nyy}[0]{n_{y}}
\newcommand{\UL}[0]{\mat{U}_{L}}
\newcommand{\UR}[0]{\mat{U}_{R}}
\newcommand{\diag}[0]{\textrm{diag}}

\newcommand{\E}[0]{\mat{E}}

\newcommand{\LamxL}[0]{\Lambda_{\xi L}}
\newcommand{\LamyL}[0]{\Lambda_{\eta L}}
\newcommand{\LamxR}[0]{\Lambda_{\xi R}}
\newcommand{\LamyR}[0]{\Lambda_{\eta R}}

\newcommand{\Mg}[0]{\mat{H}_{g}}

\usepackage{cancel}

\newcommand{\lamxi}[0]{\lambda_{\xi}}
\newcommand{\lameta}[0]{\lambda_{\eta}}

\newcommand{\Omhat}[0]{\hat{\Omega}}
\newcommand{\Gamhat}[0]{\hat{\Gamma}}
\newcommand{\Omt}[0]{\tilde{\Omega}}
\newcommand{\Gamt}[0]{\tilde{\Gamma}}

\newcommand{\tildeML}[0]{\tilde{\mat{H}}_{L}}
\newcommand{\tildeMR}[0]{\tilde{\mat{H}}_{R}}

\newcommand{\pkgammaL}[0]{\bm{p}_{k,\Gamhat_{L}}}

\newcommand{\LamL}[0]{\mat{\Lambda}_{L}}
\newcommand{\LamR}[0]{\mat{\Lambda}_{R}}

\newcommand{\vL}[0]{\bm{v}_{L}}
\newcommand{\vR}[0]{\bm{v}_{R}}

\newcommand{\Jace}[0]{\mat{J}_{e}}
\newcommand{\ue}{\bm{u}_{e}}
\newcommand{\ui}[0]{\bm{u}_{i}}
\newcommand{\Dxie}[0]{\mat{D}_{\xi e}}
\newcommand{\Detae}[0]{\mat{D}_{\eta e}}
\newcommand{\Qxie}[0]{\mat{Q}_{\xi e}}
\newcommand{\Qetae}[0]{\mat{Q}_{\eta e}}
\newcommand{\Me}[0]{\mat{H}_{e}}
\newcommand{\Be}[0]{\mat{B}_{e}}
\newcommand{\Rei}[0]{\mat{R}_{ei}}
\newcommand{\Lame}[0]{\mat{\Lambda}_{e}}
\newcommand{\Lamxie}[0]{\mat{\Lambda}_{\xi e}}
\newcommand{\Lametae}[0]{\mat{\Lambda}_{\eta e}}
\newcommand{\Meei}[0]{\mat{M}_{ee,i}^{\lambda_{\xi}}}
\newcommand{\Mei}[0]{\mat{M}_{e,i}^{\lambda_{\xi}}}
\newcommand{\onee}[0]{\bm{1}_{e}}
\newcommand{\pki}[1]{\bm{p}_{#1}}
\newcommand{\Vp}[0]{\mat{V}}
\newcommand{\Vpxi}[0]{\mat{V}_{\xi}}
\newcommand{\Extilde}[0]{\tilde{\mat{E}}_{\xi}}
\newcommand{\W}[0]{\mat{W}}
\newcommand{\Wx}[0]{\mat{W}_{\xi}}
\newcommand{\Vtilde}[0]{\tilde{\mat{V}}}
\newcommand{\Vxtilde}[0]{\tilde{\mat{V}}_{\xi}}
\newcommand{\Sxtilde}[0]{\tilde{\mat{S}}_{\xi}}
\newcommand{\Dxtilde}[0]{\tilde{\mat{D}}_{\xi}}
 \journalname{Journal of Scientific Computing}
\begin{document}
\title{Simultaneous Approximation Terms for Multi-Dimensional Summation-by-Parts Operators\footnote{Some of the material presented in this article has also appeared in: Hicken J. E., Del Rey Fern\'andez D. C and Zingg D. W. \it{Simultaneous Approximation Terms for Multi-dimensional Summation-by-parts Operators}, AIAA Aviation Conference (2016).}}

\author{David C.\ Del Rey Fern\'andez         \and
        Jason E.\ Hicken                      \and
        David W.\ Zingg
}
\institute{David C.\ Del Rey Fern\'andez, Postdoctoral Fellow\at
              University of Toronto Institute for Aerospace Studies \\
              \email{dcdelrey@gmail.com}           
           \and
           Jason E.\ Hicken, Assistant Professor\at
           Rensselaer Polytechnic Institute \\
           \email{hickej2@rpi.edu}
           \and
           David W.\ Zingg, Professor, J. Armand Bombardier Foundation Chair in Aerospace Flight, Institute for Aerospace Studies, 4925 Dufferin St., and Associate Fellow AIAA\at
           University of Toronto Institute for Aerospace Studies \\
           \email{dwz@oddjob.utias.utoronto.ca}
}

\date{}

\maketitle
\begin{abstract}
This paper is concerned with the accurate, conservative, and stable imposition
of boundary conditions and inter-element coupling for multi-dimensional
summation-by-parts (SBP) finite-difference operators.  More precisely, the focus
is on diagonal-norm SBP operators that are not based on tensor products and are
applicable to unstructured grids composed of arbitrary elements.  We show how
penalty terms --- simultaneous approximation terms (SATs) --- can be adapted to
discretizations based on multi-dimensional SBP operators to enforce boundary and
interface conditions. A general SAT framework is presented that leads to
conservative and stable discretizations of the variable-coefficient advection
equation.  This framework includes the case where there are no nodes on the
boundary of the SBP element at which to apply penalties directly.  This is an
important generalization, because elements analogous to Legendre-Gauss
collocation, \ie without boundary nodes, typically have higher accuracy for the
same number of degrees of freedom.  Symmetric and upwind examples of the general
SAT framework are created using a decomposition of the symmetric part of an SBP
operator; these particular SATs enable the pointwise imposition of boundary
and inter-element conditions. We illustrate the proposed SATs using
triangular-element SBP operators with and without nodes that lie on the
boundary.  The accuracy, conservation, and stability properties of the resulting
SBP-SAT discretizations are verified using linear advection problems with
spatially varying divergence-free velocity fields.
\end{abstract}
%
%
\section{Introduction}
\ignore{
We are interested in high-order discretizations that obey the summation by parts
(SBP) property.  Our interest is driven by two characteristics that SBP
discretizations provide.  First, the SBP property mimics integration by parts in
such a way that it greatly facilitates the construction of high-order schemes
that are conservative and provably stable.  Second, SBP operators are not
completely determined by an underlying polynomial basis, and this provides
flexibility to optimize the operators in various ways.

This paper focuses on the first characteristic: conservative and stable
high-order SBP discretizations.  In particular, the paper addresses the stable,
accurate, and conservative imposition of boundary conditions and inter-element
coupling in the context of SBP discretizations of the variable-coefficient
advection problem.  Ultimately, our plan is to apply these SBP discretizations
to split forms of nonlinear partial-differential equations (PDEs).  Split forms
can be used, for example, to prove nonlinear entropy stability of the Euler
equations
Refs.\ \cite{Harten1983,Hughes1986,Matteo2015,Carpenter2014,Fisher2013b},
and they have been shown to improve robustness \cite{Gassner2016}.

SBP methods have predominantly been developed in the context of high-order
finite difference methods~\cite{Kreiss1974,Strand1994} where the nodal
distribution in computational space is uniform; see the review
papers~\cite{Fernandez2014,Svard2014} and the references therein. While SBP
methods have been extended in a number of ways, for example
see~\cite{Fernandez2015,DCDRF2014,Gassner2013,Carpenter1996}, the majority of
these developments have been limited to one-dimensional operators that are
applied to multi-dimensional problems using tensor-product operators in
computational space. An interesting exception is the work by Nordstr\"om
\etal~\cite{Nordstrom2003}, which presents a vertex-centered
second-order-accurate finite-volume scheme with the SBP property on unstructured
grids.

The tensor-product approach, while adequate for many applications, has
limitations when applied to complex geometries and in the context of localized,
anisotropic mesh adaptation. This has motivated our interest in generalizing SBP
operators to more general multi-dimensional subdomains, \ie elements.  Thus,
building on the generalization in~\cite{DCDRF2014}, we presented an SBP
definition in~\cite{multiSBP} (see also~\cite{AIAAversion}) that is suitable for
arbitrary, bounded subdomains with piecewise smooth, orientable boundaries.

SBP derivative operators, including those in \cite{multiSBP}, do not inherently
enforce boundary conditions or inter-element coupling. The majority of SBP-based
discretizations rely on simultaneous approximation terms
(SATs)~\cite{Carpenter1994,Carpenter1999,Nordstrom1999,Nordstrom2001b} to impose
boundary conditions, as well as inter-element coupling when the solution space
is discontinuous. SATs are terms that impose boundary data and inter-element
coupling in a weak sense and lead to stable and conservative schemes without
impacting the asymptotic order of the discretization.

In~\cite{multiSBP} we derived SATs for multidimensional diagonal-norm SBP
operators and showed that the resulting discretizations are stable for the
constant-coefficient advection equation; however, the SATs in \cite{multiSBP}
are cumbersome to apply to variable coefficient problems, let alone nonlinear
PDEs.

For SBP operators that have a sufficient number of nodes on each face, we can
adopt SATs that are equivalent to the boundary-integrated numerical flux
functions used in nodal discontinuous Galerkin methods~\cite{hesthaven:2008};
however, these SATs are not provably stable for variable-coefficient advection
problems unless exact integration is used.  Moreover, we would like to consider
multi-dimensional SBP operators that do not have nodes on their boundaries; the
nodes and norm matrix of a diagonal-norm SBP operator define a strong cubature
rule~\cite{multiSBP}, and cubatures without boundary nodes tend to be more
accurate than rules with boundary nodes for the same number of
nodes~\cite{Cools1999}.

In light of the limitations of the SATs used in \cite{multiSBP}, the objectives
of the present work are to:
\begin{enumerate}
\item develop SATs that lead to provably stable and conservative schemes for
  variable-coefficient PDEs, and are extendable to nonlinear PDEs in split form,
  and;

\item generalize the SAT definition to accommodate multi-dimensional SBP
  operators that have few, if any, boundary nodes.
\end{enumerate}

The remainder of the paper is organized as follows.  After introducing some
notation, Section~\ref{sec:SBP} reviews the definition of multi-dimensional SBP
operators from~\cite{multiSBP}. Section~\ref{sec:decomp} demonstrates the
decomposition of the symmetric component of the SBP-derivative operator by
considering a set of auxiliary nodes on the boundary using
interpolation/extrapolation operators and face cubature rules. In
Sections~\ref{sec:var} and~\ref{sec:concrete} the general framework for
construction of stable and conservative SATs is presented and two examples of
SATs are discussed.  In order to illustrate SATs on a concrete example,
Section~\ref{sec:triSATs} presents two families of SBP operators for the
triangle and describes how SATs are constructed for these operators.  A
numerical verification using the linear advection equation with a spatially
varying divergence-free velocity field is given in Section~\ref{sec:results},
and conclusions are provided in Section~\ref{sec:conclude}.
}

We are interested in high-order discretizations that obey the summation by parts (SBP) property.  The SBP property mimics integration by parts, and it greatly facilitates the construction of high-order schemes that are conservative and provably stable (linearly and nonlinearly) \ignore{We are interested in schemes with the summation by parts (SBP) property, which are high-order methods that are mimetic of integration by parts, as they can be used to construct methods that are conservative and provably linearly and nonlinearly stable} \cite{Matteo2015,Carpenter2014,Fisher2013b}.  We are, in addition, interested in the flexibility provided by SBP operators that do not have a standard polynomial basis representation.  In principle, this flexibility can be used to optimize SBP operators in various ways.  For example, their efficiency can be improved by reducing their spectral radius or decreasing the number of floating point operations per node.

SBP methods have predominantly been developed in the context of high-order finite difference methods~\cite{Kreiss1974,Strand1994} where the nodal distribution in computational space is uniform; see the review papers~\cite{Fernandez2014,Svard2014} and the references therein. While SBP methods have been extended in a number of ways, for example see~\cite{Fernandez2015,DCDRF2014,Gassner2013,Carpenter1996}, the majority of these developments have been limited to one-dimensional operators that are applied to multi-dimensional problems using tensor-product operators in computational space. An interesting exception is the work by Nordstr\"om \etal~\cite{Nordstrom2003}, which presents a vertex-centered second-order-accurate finite-volume scheme with the SBP property on unstructured grids.

The tensor-product approach, while adequate for many applications, has limitations when applied to complex geometries and in the context of localized, anisotropic mesh adaptation. This motivates our interest in generalizing SBP operators to more general multi-dimensional subdomains, \ie elements.

Building on the generalization in~\cite{DCDRF2014}, we presented an SBP
definition in~\cite{multiSBP} (see also~\cite{AIAAversion}) that is suitable for
arbitrary, bounded subdomains with piecewise smooth, orientable boundaries.  For
diagonal-norm\footnote{The norm matrix can be viewed as a mass matrix.}
multi-dimensional SBP operators that are exact for polynomials of total degree
$p$, it was shown that the norm and corresponding nodes define a strong cubature
rule that is exact for polynomials of degree $2p-1$.  This connection to
cubature rules greatly simplifies the construction of SBP operators, since many
suitable cubature rules have already been identified in the
literature~\cite{Cools1999}.  In this paper, we will only consider diagonal-norm
operators.

SBP derivative operators do not inherently enforce boundary conditions or
inter-element coupling. The majority of SBP-based discretizations rely on
simultaneous approximation terms
(SATs)~\cite{Carpenter1994,Carpenter1999,Nordstrom1999,Nordstrom2001b} to impose
boundary conditions, as well as inter-element coupling when the solution space
is discontinuous. SATs are terms that impose boundary data and inter-element
coupling in a weak sense and lead to stable and conservative schemes without
impacting the asymptotic order of the discretization.

\ignore{
{\color{blue}
In this paper, we emphasize construction of SATs for split-forms of PDEs. Such split forms can be used to prove
nonlinear stability, for example of the Burger's equation and Euler equations, and more generally are related to
the entropy stability theory in Refs.\ \cite{Harten1983,Hughes1986,Matteo2015,Carpenter2014,Fisher2013b}. Moreover, such split-forms can be advantageous in improving robustness \cite{Gassner2016}.
n the context of variable coefficient problems, skew-symmetric splittings have previously been used by Kopriva and Gassner \cite{} to prove stability of the discontinuous Galerkin spectral element method and in this paper, we leverage their work in proving the stability of our semi-discretization.
}
}

In~\cite{multiSBP} we derived SATs for multidimensional diagonal-norm SBP operators and showed that the resulting discretizations are stable for the linear constant-coefficient advection equation.  Indeed, for constant-coefficient advection these penalties are the strong-form equivalent of the boundary-integrated numerical flux functions used in~\cite{hesthaven:2008}. The SATs described in \cite{multiSBP} can theoretically accommodate variable-coefficient advection problems; however, they are not practical for this class of problem because new SBP operators would be needed whenever the variable coefficients change.

The multi-dimensional SBP operators in~\cite{multiSBP} were designed to have a unisolvent set of nodes on each face for the appropriate space of polynomials.  This constraint was imposed, in part, to simplify the construction of pointwise SATs, but it increases the total number of nodes required for the SBP cubature.  For example, the quadratic, cubic, and quartic SBP operators for the triangle require 7, 12, and 18 nodes, respectively, rather than the 6, 10, and 15 nodes necessary for a total-degree basis~\cite{multiSBP}.  A similar trend is observed for tetrahedral elements.  Given the quadratic complexity of matrix-vector multiplication, there is impetus to minimize the number of volume nodes.  In addition, it is well known that strong cubature rules without boundary nodes tend to be more accurate than rules with boundary nodes for the same number of nodes~\cite{Cools1999}. \ignore{In addition, it is known that tensor-product spectral-element discretizations that use Legendre-Gauss collocation (no nodes on the element boundaries) outperform discretizations based on Legendre-Gauss-Lobatto collocation (nodes on element boundaries)~\cite{???}.}

In light of the limitations of the SATs used in \cite{multiSBP}, the objectives
of the present work are to:
\begin{enumerate}
\item generalize the SAT definition to accommodate multi-dimensional SBP
operators that may not have a sufficient number of boundary nodes to construct adequate face cubature rules, including operators that have no boundary nodes, and;
\item develop SATs that lead to provably stable and conservative schemes for variable coefficient partial differential equations (PDEs) {\color{blue} in split form}.
\end{enumerate}

The remainder of the paper is organized as follows.  After introducing some
notation, Section~\ref{sec:SBP} reviews the definition of multi-dimensional SBP
operators from~\cite{multiSBP}. Section~\ref{sec:decomp} demonstrates the decomposition of the symmetric
component of the SBP-derivative operator by considering
 a set of auxiliary nodes on the boundary using interpolation/extrapolation operators
and face cubature rules. In Sections~\ref{sec:var} and~\ref{sec:concrete}
the general framework for construction of stable and conservative SATs is presented and two examples of SATs are discussed.
 In order to
illustrate SATs on a concrete example, Section~\ref{sec:triSATs} presents two
families of SBP operators for the triangle and describes how SATs are
constructed for these operators.  A numerical verification using the linear advection
equation with a spatially varying divergence-free velocity field is given in Section~\ref{sec:results},
and conclusions are provided in Section~\ref{sec:conclude}.
%
%

\section{Notation and review of multi-dimensional summation-by-parts operators}\label{sec:SBP}

This work builds on \cite{multiSBP}, so similar notation is maintained for
consistency.  As in that work, we focus on two-dimensional operators to simplify the presentation.  One notable difference between the present work and
\cite{multiSBP} is that we consider general (smooth) bijective mappings from
physical to computational space. Domains and their boundaries in physical space
are denoted with $\Omega$ and $\Gamma$, respectively.  The corresponding sets in
computational space are given by $\Omhat$ and $\Gamhat$.  Physical-space
coordinates are represented with $(x,y) \in \Omega$, and the computational-space
coordinates are given by $(\xi,\eta) \in \Omhat$.  Several definitions and theorems are
limited to operators defined in the $\xi$ coordinate direction, since operators
defined in the other directions are analogous.

Functions are denoted with capital letters with a script type; \eg,
$\fnc{U}(\xi,\eta,t) \in L^{2}(\Omhat \times [0,T])$ denotes a square-integrable
function on the space-time domain $\hat{\Omega}\times[0,T]$.
Functions and operators are discretized on a set of $n$ nodes,
$S_{\hat{\Omega}}=\left\{(\xi_{i},\eta_{i})\right\}_{i=1}^{n} \subset \hat{\Omega}$.  The restriction of a function
to the nodes is a column vector that is represented using lower-case bold font. For example, in the case of $\fnc{U}$ we would write
\begin{equation*}
\bm{u} = \left[\fnc{U}(\xi_{1},\eta_{1}),\dots,\fnc{U}(\xi_{n},\eta_{n})\right]\Tr.
\end{equation*}

A number of definitions and theorems rely on the monomial basis,
defined below in (partial) order of nondecreasing degree.
\begin{equation*}
  \fpk(\xi,\eta) \equiv \xi^{i}\eta^{j-i}, \qquad
  k = j(j+1)/2 +i+1, \quad
  \forall\; j \in \{ 0,1,\ldots,p\}, \quad i \in \{0,1,\ldots,j\}.
\end{equation*}
The cardinality of the monomial basis of total degree $p$ is denoted
\begin{equation*}
  \nmin{p} \equiv \binom{p+d}{d},
\end{equation*}
where $d$ is the spatial dimension; for $d=2$ this gives $\nmin{p} =
(p+1)(p+2)/2$.  The monomials and their derivatives evaluated at the nodes are
represented by the $n$-vectors
\begin{align*}
  \pk &\equiv \left[ \fpk(\xi_{1},\eta_{1}),\dots,\fpk(\xi_{n},\eta_{n})\right]\Tr,\\
\text{and}\qquad
  \dxpk &\equiv \left[ \frac{\partial \fpk}{\partial \xi}(\xi_{1},\eta_{1}),\dots,
    \frac{\partial \fpk}{\partial \xi}(\xi_{n},\eta_{n})\right]\Tr.
\end{align*}

%
%

We can now state the following definition of a multi-dimensional SBP operator
that was proposed in~\cite{multiSBP}:

\begin{definition}\label{def:SBP}
  {\bf Two-dimensional summation-by-parts operator:} Consider an open and
  bounded domain $\Omhat\subset\mathbb{R}^{2}$ with a piecewise-smooth boundary
  $\Gamhat$.  The matrix $\Dx$ is a degree $p$ SBP approximation to the first
  derivative $\frac{\partial}{\partial \xi}$ on the nodes
  $S_{\Omhat}=\left\{(\xi_{i},\eta_{i})\right\}_{i=1}^{n}$ if
  \renewcommand{\theenumi}{\Roman{enumi}}%
  \begin{enumerate}
  \item $\Dx\pk = \dxpk,\qquad \forall\; k \in \{ 1,2,\ldots,\nmin{p} \}$;
    \label{sbp:accuracy}
  \item $\Dx = \M^{-1}\Qx$, where $\M$ is symmetric positive-definite; and \label{sbp:H}
  \item $\Qx = \Sx + \frac{1}{2}\Ex$, where $\Sx\Tr=-\Sx$, $\Ex\Tr=\Ex$, and
    $\Ex$ satisfies
    \begin{equation*}
      \pk\Tr\Ex\pM=\displaystyle\oint_{\Gamhat}\fpk\fpm \nx
      \mr{d}\Gamhat,\qquad \forall\; k,m \in \{ 1,2,\ldots,\nmin{r}\},
    \end{equation*} \label{sbp:Ex}
  \end{enumerate}
  where $r\ge p$, and $\nx$ is the $\xi$ component of
  $\bm{n}=\left[\nx,\ny\right]\Tr$, the outward pointing unit normal on
  $\Gamhat$.
\end{definition}

A diagonal-norm SBP operator is one where $\M$ is a diagonal matrix.  As mentioned in the Introduction, such diagonal-norm operators are closely linked to cubature rules and, under mild assumptions on a generalized Vandermonde matrix, the existence of a cubature implies the existence of an SBP operator~\cite{multiSBP}.  Conversely, the norm matrix $\M$ is a cubature rule satisfying
\begin{equation*}
\pk\Tr\M\pM=\int_{\Omhat}\fpk\fpm\mr{d}\Omhat,
\end{equation*}
where, at a minimum, $\M$ is of degree $2p-1$, \ie the above equality holds provided $\fpk\fpm$ is at most degree $2p-1$.  In addition, since $\M$ is symmetric positive definite, it defines the finite-dimensional norm (hence the name, norm matrix) that is a degree $2p-1$ approximation to the $L^2$ norm:
\begin{equation*}
  \| \bm{u} \|_{\M}^{2} \equiv \bm{u}\Tr \M \bm{u} \approx \int_{\Omhat} \fnc{U}^{2} \mr{d}\Omhat.
\end{equation*}
We use this norm frequently in the following stability analysis.

\ignore{
\todo[inline]{Jason: I removed this because we do not use it}
Finally, it was shown in~\cite{multiSBP} that the matrices $\Qx$ and $\Sx$ satisfy
\begin{align*}
\pk\Tr \Qx \pM &= \int_{\Omhat} \fpk \dxfpm \mr{d} \Omhat, \\
 \text{and}\qquad
\pk\Tr \Sx \pM &= \int_{\Omhat} \fpk \dxfpm \mr{d}\Omhat
- \frac{1}{2}\oint_{\Gamhat} \fpk \fpm \nx \mr{d}\Gamhat,
\end{align*}
where $k,m \leq \nmin{r}$, and $\fpk \fpm$ is at most degree $2p$.
}
%
%
\section{Decomposition of $\Ex$}\label{sec:decomp}

The pointwise nature of SATs complicates their direct application to
multi-dimensional SBP operators, which may not have any nodes on the boundary of
their domain.  Fortunately, as we show in this section, the $\Ex$ and $\Ey$
matrices of many multi-dimensional SBP operators can be decomposed in terms of
nodes that lie on the boundary of $\Omhat$.  These auxiliary nodes provide a
straightforward means of applying traditional SAT penalties.

In order to proceed, we introduce some assumptions regarding the reference
element, $\Omhat$, and its boundary, $\Gamhat$.

\begin{assume}\label{assume:Gamma}
  The reference element $\Omhat$ is a polygon, and its boundary $\Gamhat$ is
  piecewise linear with $\Gamhat = \bigcup_{j=1}^{\kappa} \Gamhat_{j}$ and
  $\bigcap_{j=1}^{\kappa} \Gamhat_{j} = \emptyset$.  Furthermore, for each
  $\Gamhat_{j}$ there exists a strong cubature rule, with nodes $S_{\Gamhat_{j}}
  = \{ (\xi_{i}^{(j)},\eta_{i}^{(j)}) \}_{i=1}^{n_{j}}$ and weights
  $\{b_{i}^{(j)}\}_{i=1}^{n_{j}}$, that exactly integrates polynomial integrands
  of degree $q \geq 2r$, where $r$ is the integer used in property~\ref{sbp:Ex}
  of Definition~\ref{def:SBP}.
\end{assume}

\begin{remark}
  The assumption that the reference element is a polygon is reasonable; for
  example, the most common finite elements are polytopes in computational space.
\end{remark}

Assumption~\ref{assume:Gamma} ensures that we can evaluate element boundary
fluxes with sufficient accuracy in computational space.  For example, the
cubature rule for each $\Gamhat_{j}$ allows us to write
\begin{equation*}
  \int_{\Gamhat_{j}} \fpk(\xi,\eta) \nx \,\mr{d}\Gamhat =
  n_{\xi j} \sum_{i=1}^{n_{j}} b_{i}^{(j)} \fpk\left(\xi_{i}^{(j)},\eta_{i}^{(j)}\right),
\end{equation*}
for all polynomials $\fpk$ of total degree $q$ or less, where $q \geq 2r$. A
similar expression holds for fluxes in the $\eta$ direction.  Note that $n_{\xi
  j}$ is constant over $\Gamhat_{j}$, due to linearity, so it can be pulled
outside the integral.

Our approach to imposing boundary and interface conditions pointwise is to
interpolate/extrapolate the solution from the element nodes onto the cubature
nodes $S_{\Gamhat_{j}} = \{(\xi_{i}^{(j)},\eta_{i}^{(j)}\}_{i=1}^{n_{j}}$ using
operators $\mat{R}_{j}$.  To decompose $\Ex$ and $\Ey$ , these
interpolation/extrapolation operators must be sufficiently accurate.
Specifically,
\begin{equation*}
\left(\R_j \pk\right)_{i} = \fpk\left(\xi_{i}^{(j)},\eta_{i}^{(j)}\right), \quad i=1,2,\ldots,n_{j}, \qquad \forall \; k \in \{1,2,\ldots,\nmin{r}\},
\end{equation*}
where $r \geq p$.

We first prove that we can construct $\Ex$ that satisfy the requirements of
Definition \ref{def:SBP} using the interpolation/extrapolation operators,
$\mat{R}$, and the face-cubature rules from Assumption~\ref{assume:Gamma}.

\begin{thrm}\label{thm:SATdecomp}
  Let Assumption~\ref{assume:Gamma} hold and let
  $S_{\Omhat}=\left\{\left(\xi_{i},\eta_{i}\right)\right\}_{i=1}^{n}$ be a given
  a nodal distribution on the domain $\Omhat$.  Then, a matrix $\Ex$ that
  satisfies the requirements of Definition \ref{def:SBP} can be constructed as
  \begin{equation}
    \Ex = \sum_{j=1}^{\kappa} n_{\xi j} \R_{j}\Tr \B_{j} \R_{j}, \label{eq:SATdecomp}
  \end{equation}
  where $\B_{j} =
  \mydiag\left(b_{1}^{(j)},b_{2}^{(j)},\ldots,b_{n_{j}}^{(j)}\right)$ is a
  diagonal matrix holding the cubature weights for $\Gamhat_{j}$, and $\R_{j}
  \in \mathbb{R}^{n_{j} \times n}$ is a degree $r \geq p$
  interpolation/extrapolation operator from the nodes $S_{\Omhat}$ to the nodes
  of the reference boundary domain, $S_{\Gamhat_{j}}$.
\end{thrm}

\begin{proof}
  The terms $\R_{j}\Tr \B_{j} \R_{j}$ are symmetric by construction.  Therefore,
  we need only show that the accuracy conditions of Property~\ref{sbp:Ex} hold.
  Since the $\R_{j}$ are exact for degree $r \geq p$ polynomials, we have,
  $\forall\; k,m \in \{ 1,2,\ldots,\nmin{r} \}$,
  \begin{align*}
    \pk\Tr\Extilde\pM &= \sum_{j=1}^{\kappa} n_{\xi j} \pk\Tr \R_{j}\Tr \B_{j} \R_{j} \pM \\
    &= \sum_{j=1}^{\kappa} n_{\xi j} \sum_{i=1}^{n_{j}} b_{i}^{(j)} \fpk\left(\xi_{i}^{(j)},\eta_{i}^{(j)}\right)
    \fpm\left(\xi_{i}^{(j)},\eta_{i}^{(j)}\right) \\
    &= \sum_{j=1}^{\kappa}
    \int_{\Gamhat_{j}} \fpk\fpm n_{\xi j} \, \mr{d}\Gamhat,
  \end{align*}
  where we have used the linearity of the mappings (see
  Assumption~\ref{assume:Gamma}) and the fact that the product $\fpk\fpm$ has
  total degree less than or equal to $2r\leq q$.  The result follows by the
  additive property of integrals.\qed
\end{proof}

Next, we prove that we can construct a multi-dimensional SBP operator $\Dx$ from
a given, sufficiently accurate, strong cubature rule and an $\Ex$ matrix that
satisfies the requirements of Definition~\ref{def:SBP}.  For this result, we
will need the degree $p$ (rectangular) Vandermonde matrix
\begin{equation*}
\Vp \equiv \left[\pki{1},\pki{2},\dots,\pki{\nmin{p}}\right],
\end{equation*}
as well as the associated matrix containing the projection of the $\xi$
derivatives of the monomials
\begin{equation*}
\Vpxi\equiv\left[\pki{1}',\pki{2}',\dots,\pki{\nmin{p}}'\right].
\end{equation*}

\begin{thrm}\label{thm:exists}
  Let the diagonal elements of $\M$ and the nodes
  $S_{\Omhat}=\left\{\left(\xi_{i},\eta_{i}\right)\right\}_{i=1}^{n}$ define a
  degree $2p-1$ strong cubature rule on the domain $\Omhat$.  If the Vandermonde
  matrix, $\Vp$, associated with the nodes $S_{\Omhat}$ has full column rank, and
  $\Ex \in \mathbb{R}^{n\times n}$ is symmetric and satisfies
  Property~\ref{sbp:Ex} of Definition~\ref{def:SBP}, then there exists at least
  one degree $p$ SBP operator, $\Dx = \M^{-1}(\Sx + \frac{1}{2}\Ex)$, based on
  the given nodes $S_{\Omhat}$ and matrices $\M$ and $\Ex$.
\end{thrm}

\begin{proof}
  We must show that, for the given $\M$ and $\Ex$, we can find a skew symmetric
  $\Sx$ that satisfies the accuracy conditions~\ref{sbp:accuracy}.  The SBP
  accuracy conditions can be recast as the following set of matrix conditions:
  \begin{equation}\label{eq:thrm2accuracy}
    \Dx\overbrace{\left[\Vp\; \W\right]}^{\Vtilde} = \overbrace{\left[\Vp\; \Wx\right]}^{\Vxtilde},
  \end{equation}
  where $\W$ is a set of linearly independent vectors, of size
  $n\times(n-\nmin{p})$, such that $\Vtilde$ is invertible, \eg, a basis for the
  null space of $\Vp$.  The matrix $\Wx$ is of size $n\times(n-\nmin{p})$ and is
  to be determined.  Using \eqref{eq:thrm2accuracy} and the multi-dimensional
  SBP definition, we can solve for $\Sx$ as
  \begin{equation}\label{eq:thrm2step0}
    \Sx = \M\Vxtilde\Vtilde^{-1}-\frac{1}{2}\Ex.
  \end{equation}

  What remains to be shown is that $\Sx$ can be constructed to be skew symmetric
  using $\Wx$.  Rather than doing so directly, we show that an associated matrix
  is skew symmetric; left and right multiplying \eqref{eq:thrm2step0} by
  $\Vtilde\Tr$ and $\Vtilde$, respectively, results in
  \begin{equation}\label{eq:thrm2step1}
    \Sxtilde \equiv \Vtilde\Tr\Sx\Vtilde=
    \begin{bmatrix}
      \V\Tr\M\Vx-\frac{1}{2}\V\Tr\Ex\V & \V\Tr\M\Wx-\frac{1}{2}\V\Tr\Ex\W \\
      \W\Tr\M\Vx-\frac{1}{2}\W\Tr\Ex\V & \W\Tr\M\Wx-\frac{1}{2}\W\Tr\Ex\W
    \end{bmatrix}.
  \end{equation}
  If we can show that $\Sxtilde$ is skew symmetric, this will imply skew
  symmetry for $\Sx$.

  We first show that the block $\Sxtilde(1:\nmin{p},1:\nmin{p})$ is skew symmetric;
  \begin{equation*}
    \Sxtilde(1:\nmin{p},1:\nmin{p})+\Sxtilde(1:\nmin{p},1:\nmin{p})\Tr = 
    \V\Tr\M\Vx+\Vx\Tr\M\V-\V\Tr\Ex\V = 0,
  \end{equation*}
  where we have used the compatibility conditions $\Vp\Tr\M\Vpxi+\Vpxi\Tr\M\Vp =
  \Vp\Tr\Ex\Vp$ \cite{multiSBP,DCDRF2014,Strand1994,Kreiss1974}.
 
  Next, notice that the entries in the lower-left block
  $\Sxtilde(\nmin{p}+1:n,1:\nmin{p})$ are fully specified by $\Vp$, $\W$, and
  $\Vx$.  To make the rest of $\Sxtilde$ skew symmetric, we specify the columns
  of $\Wx$ in order to match the upper-right block to the (negative transposed)
  lower-left block and a zero lower-right block.  In other words, for
  skew-symmetry we require that
  \begin{equation*}
    \Sxtilde(:,\nmin{p}+j) = \Vtilde\Tr\M\Wx - \frac{1}{2}\Vtilde\Tr\Ex\W
    = \begin{bmatrix}
      -\Vx\Tr\M\W + \frac{1}{2} \V\Tr\Ex\W \\
      \mat{0} 
    \end{bmatrix}.
  \end{equation*}
  Rearranging we have
  \begin{equation*}
    \Vtilde\Tr\M\Wx = \frac{1}{2}\Vtilde\Tr\Ex\W + \begin{bmatrix}
      -\Vx\Tr\M\W + \frac{1}{2} \V\Tr\Ex\W \\
      \mat{0} 
    \end{bmatrix},
  \end{equation*}
  which we can solve for $\Wx$, because $\Vtilde\Tr\M$ is invertible. This
  particular $\Wx$ ensures that $\Sxtilde$ is skew symmetric, and, therefore,
  guarantees that $\Sx$ is skew symmetric. \qed
\end{proof}

Theorems~\ref{thm:SATdecomp} and~\ref{thm:exists} imply the following:
\begin{corollary}\label{lemma:exists}
  If the hypotheses of Theorem~\ref{thm:SATdecomp} and~\ref{thm:exists} are met,
  then there exists at least one degree $p$ SBP operator whose $\Ex$ has the
  decomposition~\eqref{eq:SATdecomp}.
\end{corollary}

\begin{remark}
  The implication of Corollary~\ref{lemma:exists} is that $\Dx$ exist for which
  $\Ex$ is constructed as the sum of matrices, each of which is associated with
  a face $\Gamhat_{j}$; more importantly, it expresses these constituent
  matrices in terms of the face cubature points and the $\mat{R}_{j}$ operators,
  facilitating pointwise imposition of SATs.
\end{remark}

\ignore{
We close this section by examining the interesting case of a square full rank Vanderomnde matrix constructed as
\begin{equation}
\mat{V}_{}
\end{equation}
In the case where $n = \nmin{p}$ and we have a full rank Vandermonde matrix $\Vp$, which occurs if we have a full polynomial basis, then 
we have the following result:

\begin{thrm}\label{thm:unique}
Given a nodal distribution $S_{\Omhat}\left\{\left(\xi_{i},\eta_{i}\right)\right\}_{i=1}^{n}$, for which $n=\nmin{p}$, 
and the Vandermonde matrix $\Vp$ is full rank, 
then there exists a unique difference operator $\Dxtilde$ of degree $p$.
\end{thrm}
\begin{proof}
The accuracy conditions on $\Dxtilde$ are the same as for a degree $p$ SBP operator (see Definition~\ref{def:SBP} conditions ~\ref{sbp:accuracy})
 therefore, $\Dxtilde$ is uniquely give by
\begin{equation*}
\Dxtilde = \Vx\Vp^{-1}.
\end{equation*}
\end{proof}
\qed

We can now prove that on nodal distributions with $n=\nmin{p}$ that have a full rank Vandermonde matrix and a diagonal norm $\M$ of at least degree $2p-1$, 
the resulting SBP operator, while unique, has non-unique decomposition.

\begin{thrm}\label{thm:nonunique}
Let Assumption~\ref{assume:Gamma} hold on the domain $\Omhat$ for at least two surface nodal distributions $S_{\Gamhat,1}$ and $S_{\Gamhat,2}$. 
Then, if for the nodal distribution $S_{\Omhat}\left\{\left(\xi_{i},\eta_{i}\right)\right\}_{i=1}^{n}$,
$n=\nmin{p}$, the Vandermonde matrix $\Vp$ is full rank, and there 
exists a diagonal-norm matrix $\M$, of at least degree $2p-1$, then the unique SBP operator, $\Dx$, that exists has a non-unique decomposition. 
\end{thrm}
\begin{proof}
Since Assumption~\ref{assume:Gamma} hold for two surface nodal distributions $S_{\Gamhat,1}$ 
and $S_{\Gamhat,2}$ we can construct $\mat{E}_{\xi,1}$ and $\mat{E}_{\xi,2}$ as per Theorem~\ref{thm:SATdecomp}. By Theorem~\ref{thm:exists} two SBP operators exist, 
however, by Theorem~\ref{thm:unique} they must be the same.\qed
\end{proof}
}
\ignore{
We are now in a position to state and prove the main result of this section.

\begin{thrm}\label{thm:SATdecomp}
  Let Assumption~\ref{assume:Gamma} hold and let $\Dx = \M^{-1}\Qx$ be a degree $p$ SBP approximation of the first partial derivative with respect to $\xi$ on the domain $\Omhat$.  Then the symmetric part of $\Qx$ can be written as
  \begin{equation}
    \frac{1}{2}\Ex = \frac{1}{2} \sum_{j=1}^{\kappa} n_{\xi j} \R_{j}\Tr \B_{j} \R_{j}, \label{eq:SATdecomp}
  \end{equation}
  where $\B_{j} = \mydiag\left(b_{1}^{(j)},b_{2}^{(j)},\ldots,b_{n_{j}}^{(j)}\right)$ is a diagonal matrix holding the cubature weights for $\Gamhat_{j}$, and $\R_{j} \in \mathbb{R}^{n_{j} \times n}$ is a
  degree $r \geq p$ interpolation/extrapolation operator from the nodes
  $S_{\Omhat}$ to the nodes of the reference boundary
  domain, $S_{\Gamhat_{j}}$.
\end{thrm}

\begin{proof}
  The terms $\R_{j}\Tr \B_{j} \R_{j}$ are symmetric by
  construction.  Therefore, we need only show that the accuracy
  conditions of Property~\ref{sbp:Ex} hold.  Since the $\R_{j}$ are exact for
  degree $r \geq p$ polynomials, we have, $\forall k,m \in \{ 1,2,\ldots,\nmin{r} \}$,
  \begin{align*}
    \pk\Tr\Ex\pM &= \sum_{j=1}^{\kappa} n_{\xi j} \pk\Tr \R_{j}\Tr \B_{j} \R_{j} \pM \\
    &= \sum_{j=1}^{\kappa} n_{\xi j} \sum_{i=1}^{n_{j}} b_{i}^{(j)} \fpk\left(\xi_{i}^{(j)},\eta_{i}^{(j)}\right)
    \fpm\left(\xi_{i}^{(j)},\eta_{i}^{(j)}\right) \\
    &= \sum_{j=1}^{\kappa}
    \int_{\Gamhat_{j}} \fpk\fpm n_{\xi j} \, \mr{d}\Gamhat,
  \end{align*}
  where we have used the linearity of the mappings (see
  Assumption~\ref{assume:Gamma}) and the fact that the product $\fpk\fpm$ has
  total degree less than or equal to $2r\leq q$.  The result follows by the
  additive property of integrals.\qed
\end{proof}
}
%
%
\section{Linear variable-coefficient PDEs}\label{sec:var}
\ignore{
While the ideas presented in this paper can be used to construct SATs for PDEs in various forms,
one of our goals is to develop SATs suitable for split forms of nonlinear PDEs.
Therefore, in this section, we develop a general framework for constructing SATs to impose
boundary and inter-element conditions weakly for multi-dimensional SBP
operators in split form. We do so by examining the variable-coefficient linear advection
equation in two dimensions in skew-symmetric form. We review the use of the
energy method to prove stability for the continuous problem and then apply the
same ideas to prove stability of the semi-discrete equations. The goal of this
section is to determine a set of conditions that, if satisfied, lead to stable
and conservative semi-discrete forms.
}
While the ideas presented in this paper can be used to construct SATs for PDEs in various forms,
one of our goals is to develop SATs suitable for split forms of nonlinear PDEs. Such split forms can be used to prove
nonlinear stability, for example of the Burgers and Euler equations, and more generally are related to
the entropy stability theory in Refs.\ \cite{Harten1983,Hughes1986,Matteo2015,Carpenter2014,Fisher2013b}. Moreover, such split-forms can be advantageous in improving robustness \cite{Gassner2016}.
Therefore, in this section, we develop a general framework for constructing SATs to impose
boundary and inter-element conditions weakly for multi-dimensional SBP
operators in split form. We do so by examining the variable-coefficient linear advection
equation in two dimensions in skew-symmetric form. We review the use of the
energy method to prove stability for the continuous problem and then apply the
same ideas to prove stability of the semi-discrete equations. The goal of this
section is to determine a set of conditions that, if satisfied, lead to stable
and conservative semi-discrete forms.
%
%
\subsection{Stability and conservation of the variable-coefficient linear advection equation}
Consider the divergence and skew-symmetric forms of the linear advection equation with a spatially
varying velocity field, $\veclambda =\left[\lamxx,\lamyy\right]\Tr$:
\begin{equation}\label{eq:LCvar}
\begin{split}
\frac{\partial\fnc{U}}{\partial t}=&-\vecnabla\cdot\left(\veclambda \fnc{U} \right)\\
=&-\frac{1}{2}\vecnabla \cdot\left(\veclambda \fnc{U}\right)
-\frac{1}{2}\veclambda \cdot\nabla \fnc{U}
-\frac{1}{2}\fnc{U} \vecnabla\cdot\veclambda,
\end{split}
\end{equation}
for all $(x,y)\in\Omega$ and $t\ge 0$.  The initial and boundary conditions are
\begin{equation}\label{eq:icbc}
\begin{alignedat}{2}
\fnc{U}(x,y,0) &= \fnc{F}(x,y),&
\qquad &\forall\; (x,y)\in\Omega,\\
\fnc{U}(x,y,t) &= \fnc{G}(x,y,t),&
\qquad &\forall\; (x,y)\in\Gamma^{-},\quad t\ge 0,
\end{alignedat}
\end{equation}
respectively, where $\Gamma$ is subdivided into the inflow boundary,
\begin{equation*}
\Gamma^{-}=\left\{(x,y)\in\Gamma\,|\, \nxx\lamxx+\nyy\lamyy<0\right\},
\end{equation*}
and the outflow boundary, $\Gamma^{+}=\Gamma\backslash\Gamma^{-}$.

We use the energy method --- applying a similar analysis to that in Ref.~\cite{Kopriva2014} to our two-dimensional problem --- to prove that the problem defined by \eqref{eq:LCvar} and
\eqref{eq:icbc} is stable; those interested in further details are referred to
the texts~\cite{Gustafsson2013,Kreiss2004}.  Multiplying the divergence form of \eqref{eq:LCvar} by
$\fnc{U}$ and integrating in space results in
\begin{equation}\label{eq:energy1}
\int_{\Omega}\fnc{U}\frac{\partial\fnc{U}}{\partial t}\mr{d}\Omega=
-\int_{\Omega}\left(\fnc{U}\frac{\partial\lamxx\fnc{U}}{\partial x}+
\fnc{U}\frac{\partial\lamyy\fnc{U}}{\partial y}\right)\mr{d}\Omega.
\end{equation}
Equation (\ref{eq:energy1}) can be expressed in two alternative forms: i) using
the Leibniz rule on the left-hand side and the product rule on the right-hand
side, or; ii) using the Leibniz rule on the left-hand side and expanding the
derivative terms on the right-hand side. Adding these two alternative
expressions leads to
\begin{equation}\label{eq:energy4}
\frac{\mr{d}\|\fnc{U}\|^{2}}{\mr{d}t}=
-\int_{\Omega}\left(\frac{\partial\lamxx\fnc{U}^{2}}{\partial x}+
\frac{\partial\lamyy\fnc{U}^{2}}{\partial y}+
\fnc{U}^{2}\frac{\partial\lamxx}{\partial x}+
\fnc{U}^{2}\frac{\partial\lamyy}{\partial y}
\right)\mr{d}\Omega.
\end{equation}
Using integration by parts on the first two terms of the right-hand side of
(\ref{eq:energy4}) leads to
\begin{equation}\label{eq:energy5}
\frac{\mr{d}\|\fnc{U}\|^{2}}{\mr{d}t}=
-\oint_{\Gamma}\fnc{U}^{2}\left(\nxx\lamxx+\nyy\lamyy\right)\mr{d}\Gamma-
\int_{\Omega}
\fnc{U}^{2}\left(\frac{\partial\lamxx}{\partial x}+
\frac{\partial\lamyy}{\partial y}\right)
\mr{d}\Omega.
\end{equation}
Assuming that the divergence of $\veclambda$ is bounded, that is
\begin{equation*}
\alpha = \max\limits_{(x,y)\in\Omega}\left(\frac{\partial\lamxx}{\partial x}+\frac{\partial\lamyy}{\partial y}\right) \leq \infty,
\end{equation*}
we have that
\ignore{
Assuming that the divergence of $\veclambda$ is square-integrable on
$\Omega$ --- specifically, that $\|\vecnabla \cdot \veclambda \| \leq \alpha$
for some finite $\alpha \geq 0$ --- we have that
}
\begin{equation*}
\int_{\Omega}
\fnc{U}^{2}\left(\frac{\partial\lamxx}{\partial x}+
\frac{\partial\lamyy}{\partial y}\right)
\mr{d}\Omega \leq \alpha\|\fnc{U}\|^{2}.
\end{equation*}

The following inequality results from breaking up the surface integral in (\ref{eq:energy5}) into integrals over $\Gamma^{+}$ and $\Gamma^{-}$,
inserting the boundary condition, and making use of the above bound:
\begin{equation}\label{eq:energy6}
\begin{split}
\frac{\mr{d}\|\fnc{U}\|^{2}}{\mr{d}t}\leq&-\oint_{\Gamma^{+}}\fnc{U}^{2}|\nxx\lamxx+\nyy\lamyy|\mr{d}\Gamma+\oint_{\Gamma^{-}}\fnc{G}^{2}|\nxx\lamxx+\nyy\lamyy|\mr{d}\Gamma+\alpha\|\fnc{U}\|^{2}\\
\leq&\oint_{\Gamma^{-}}\fnc{G}^{2}|\nxx\lamxx+\nyy\lamyy|\mr{d}\Gamma+\alpha\|\fnc{U}\|^{2}.
\end{split}
\end{equation}
We integrate (\ref{eq:energy6}) in time and apply the initial condition
(see~\cite{Gustafsson2013} pg.\ 94), resulting in the estimate
\begin{equation*}
\begin{split}
\|\fnc{U}\|^{2}\leq\exp\left(\alpha t\right)\|\fnc{F}\|^{2}+\int_{0}^{t}\exp\left(\alpha(t-\tau)\right)\beta(\tau)\mr{d}\tau, \\
\textrm{ where }\qquad \beta \equiv \oint_{\Gamma^{-}}\fnc{G}^{2}|\nxx\lamxx+\nyy\lamyy|\mr{d}\Gamma.
\end{split}
\end{equation*}
Thus, we see that the problem defined by (\ref{eq:LCvar}) and
(\ref{eq:icbc}) is stable in the sense of Hadamard, that is, the solution
depends continuously on the data~\cite{Gustafsson2013}.

\begin{remark}
In the numerical experiments presented later, we consider the special case
of \eqref{eq:LCvar} where the velocity is divergence free,
$\vecnabla\cdot\veclambda =0$, and the boundary conditions are periodic.
Under these assumptions, starting from \eqref{eq:energy4}, the energy method results in
\begin{equation*}
\frac{\mr{d}\|\fnc{U}\|^{2}}{\mr{d}t}=0,
\end{equation*}
which shows that the energy is constant.
\end{remark}

In addition to stability, we are interested in constructing schemes that
are conservative. To understand the discrete conditions that will be imposed, it
is useful to delineate the conditions on the continuous problem.  The PDE
(\ref{eq:LCvar}) has an integral form representation given as
\begin{equation}\label{eq:con}
\frac{\mr{d}}{\mr{d}t}\int_{\Omt}\fnc{U}\mr{d}\Omega+\oint_{\Gamt}\fnc{U}\bm{n}\cdot \bm{\lambda}\mr{d}\Gamma = 0,
\end{equation}
where \eqref{eq:con} holds for all arbitrary subdomains $\Omt \subset \Omega$
with piecewise smooth, orientable boundaries $\Gamt$. Typically, the strong form
\eqref{eq:LCvar} is discretized using SBP-SAT schemes; nevertheless, we would
like our schemes to mimic \eqref{eq:con} discretely on arbitrary domains
composed of one or more elements.

%
%

\subsection{The generic SBP-SAT semi-discretization}\label{sec:coord}

In this section, we present a generic SBP-SAT semi-discretization of
\eqref{eq:LCvar}, and then determine the general conditions on the SATs necessary to
obtain an energy-stable, accurate, and conservative scheme.  We focus on SATs
for inter-element coupling --- weak enforcement of boundary conditions is similar.

The domain $\Omega$ is partitioned into $E$ nonoverlapping elements: $\Omega = \bigcup_{e = 1}^{E}\Omega_{e}$.  On each
element $\Omega_{e}$, the PDE \eqref{eq:LCvar} is mapped from physical coordinates to
computational, or reference, coordinates. For a time-invariant transformation, this results in the following skew-symmetric form:
\begin{equation}\label{eq:curvvarPDE}
\frac{\partial\fnc{J}\fnc{U}}{\partial t}
+\frac{1}{2}\frac{\partial\lamx\fnc{U}}{\partial \xi}
+\frac{1}{2}\frac{\partial\lamy\fnc{U}}{\partial \eta}
+\frac{\lamx}{2}\frac{\partial\fnc{U}}{\partial \xi}
+\frac{\lamy}{2}\frac{\partial\fnc{U}}{\partial \eta}
+\frac{\fnc{U}}{2}\left(\frac{\partial\lamx}{\partial \xi}+\frac{\partial\lamy}{\partial \eta}\right)=0,
\end{equation}
where
\begin{equation*}
\lamxi =\frac{\partial y}{\partial \eta} \lamxx- \frac{\partial x}{\partial \eta}\lamyy ,\quad
\lameta = -\frac{\partial y}{\partial \xi} \lamxx + \frac{\partial x}{\partial \xi} \lamyy,\quad
\fnc{J} = \frac{\partial x}{\partial \xi} \frac{\partial y}{\partial \eta} - \frac{\partial x}{\partial \eta}\frac{\partial y}{\partial \xi}.
\end{equation*}

To present and analyze the SATs, we consider the interface between two generic
elements, labeled ``left'' and ``right,'' having solutions $\uL$ and $\uR$,
respectively; see Figure \ref{fig:illustration}.  Suppose the left and right
elements have $\kappa_{L}$ and $\kappa_{R}$ faces, respectively.  Then, without
loss of generality, we can index the non-shared faces such that the $\ExL$ and
$\ExR$ decompositions can be written as (see Theorem~\ref{thm:SATdecomp})
\begin{equation*}
\ExL = \sum_{j=1}^{\kappa_{L}-1}\ExLi{j} + \nxL\RL\Tr\B_{L}\RL
\quad\text{and}\quad
\ExR = \sum_{j=1}^{\kappa_{R}-1}\ExRi{j} + \nxR\RR\Tr\B_{R}\RR,
\end{equation*}
where the terms $\nxL\RL\Tr\B_{L}\RL$ and $\nxR\RR\Tr\B_{R}\RR$ correspond to
the shared face.  Similar expressions hold for $\EyL$ and $\EyR$.  In the
following, we will focus on the shared face and will drop contributions from the
remaining $\kappa_{L}-1$ faces on the left and $\kappa_{R}-1$ faces on the
right, unless otherwise noted.

\begin{figure*}[tp]
 \begin{center}
 \includegraphics[width=\textwidth]{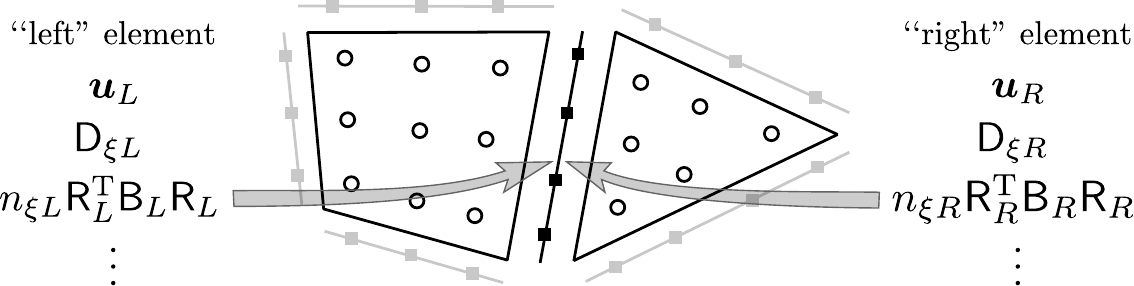}
 \caption[]{Illustration of two generic SBP elements and their common interface
   used for the analysis of SATs. The circles denote the volume nodes at which
   the solutions $\uL$ and $\uR$ are stored, and the black squares denote
   cubature nodes for the face; the latter are introduced in
   Section~\ref{sec:decomp}.\label{fig:illustration}}
 \end{center}
\end{figure*}

The SBP-SAT semi-discretization of \eqref{eq:curvvarPDE} is given by
\begin{multline}\label{eq:fullL}
\frac{\mr{d}}{\mr{d}t} \left( \JacL \uL \right)
+\frac{1}{2}\DxL\LamxL\uL
+\frac{1}{2}\DyL\LamyL\uL
+\frac{1}{2}\LamxL\DxL\uL
+\frac{1}{2}\LamyL\DyL\uL\\
+\frac{\UL}{2}\left(\DxL\LamxL\oneL+\DyL\LamyL\oneL\right)
= \underbrace{\frac{1}{2}\ML^{-1}\left(\MLL\uL-\MLR\uR\right)}_{\bm{\mr{SAT}}_{L}},
\end{multline}
on the left element and
\begin{multline}\label{eq:fullR}
\frac{\mr{d}}{\mr{d}t}\left( \JacR \uR \right)
+\frac{1}{2}\DxR\LamxR\uR
+\frac{1}{2}\DyR\LamyR\uR
+\frac{1}{2}\LamxR\DxR\uR
+\frac{1}{2}\LamyR\DyR\uR\\
+\frac{\UR}{2}\left(\DxR\LamxR\oneR+\DyR\LamyR\oneR\right)
= \underbrace{\frac{1}{2}\MR^{-1}\left(\MRR\uR-\MRL\uR\right)}_{\bm{\mr{SAT}}_{R}},
\end{multline}
on the right element.  These discretizations introduce several new matrices and vectors.
For instance,
\begin{align*}
\LamxL &= \diag\left(\lamx(\xi_{1},\eta_{1}),\dots,\lamx(\xi_{n_{L}},\eta_{n_{L}})\right),\\
\UL &= \diag\left(u_{L,1},\dots,u_{L,n_{L}}\right),\\
\mat{J}_{L} &= \diag\left(\fnc{J}(\xi_{1},\eta_{1}),\dots,\fnc{J}(\xi_{n_L},\eta_{n_L})\right),
\end{align*}
are diagonal matrices, where $n_{L}$ is the number of nodes in the left element.  Similar definitions
hold for $\LamyL$, $\LamxR$, $\LamyR$, $\UR$, and $\Jac_{R}$.  In addition, $\oneL \in \mathbb{R}^{n_L}$ and $\oneR \in \mathbb{R}^{n_R}$ are constant vectors with entries equal to one.

The matrices $\MLL \in \mathbb{R}^{n_L \times n_L}$, $\MLR \in \mathbb{R}^{n_L
  \times n_R}$, $\MRR \in \mathbb{R}^{n_R \times n_R}$, and $\MRL \in
\mathbb{R}^{n_R \times n_L}$ depend on the spatially varying field
$\veclambda_{\xi} = [\lambda_{\xi},\lambda_{\eta}]\Tr$, in general, and this
dependence is reflected in the notation. It is these four SAT matrices that we
seek to constrain using stability, accuracy, and conservation conditions.

%
%
\subsubsection{Stability}

We begin by determining the condition imposed by stability on the SAT
matrices. The energy method applied to \eqref{eq:fullL} and \eqref{eq:fullR}
consists of multiplying the equations by $\uL\Tr\ML$ and $\uR\Tr\MR$,
respectively, and adding the resulting expressions.  After cancellation, and
using the fact that $\bm{u}\Tr\mat{H}\frac{\mr{d}\Jac
  \bm{u}}{\mr{d}t}=\frac{1}{2}\frac{\mr{d}\bm{u}\Tr\mat{H}\Jac\bm{u}}{\mr{d}t}$,
we find
\begin{multline}\label{eq:energyLR}
\frac{\mr{d}}{\mr{d}t}\left(\uL\Tr\ML\JacL\uL + \uR\Tr\MR\JacR\uR \right)= \\
-\uL\Tr\UL\ML\left(\DxL\LamxL\oneL+\DyL\LamyL\oneL\right)
 -\uR\Tr\UR\MR\left(\DxR\LamxR\oneR+\DyR\LamyR\oneR\right) \\
-\uL\Tr\RL\Tr\BL\RL\LamL\uL + \uL\Tr\MLL\uL - \uL\Tr\MLR\uR \\
-\uR\Tr\RR\Tr\BR\RR\LamR\uR + \uR\Tr\MRR\uR - \uR\Tr\MRL\uL,
\end{multline}
where
\begin{equation*}
\LamL = \nxL\LamxL + \nyL\LamyL\qquad\text{and}\qquad
\LamR = \nxR\LamxR + \nyR\LamyR.
\end{equation*}
As explained earlier, we have retained only those boundary matrices
corresponding to the common face shared by the left and right elements.  The
terms corresponding to the remaining faces have been dropped to simplify the
presentation.

We treat the first terms on the right, \ie the terms on the second line of
\eqref{eq:energyLR}, in much the same way as we did for the continuous analysis.
In particular, assuming the coordinate transformation is differentiable and has a bounded and nonzero Jacobian, we have
\begin{align*}
-\uL\Tr\UL\ML\left(\DxL\LamxL\oneL+\DyL\LamyL\oneL\right) &\leq \gamma_{L}\|\uL\|_{\tildeML}^{2},\\
-\uR\Tr\UR\MR\left(\DxR\LamxR\oneR+\DyR\LamyR\oneR\right) &\leq \gamma_{R}\|\uR\|_{\tildeMR}^{2}
\end{align*}
where
\begin{equation*}
\begin{split}
&\gamma_{L}=\max\limits_{i\in[1,n_{L}]}\left[\JacL^{-1}\left(\DxL\LamxL\oneL+\DyL\LamyL\oneL\right)\right]_{i},\\
&\gamma_{R}=\max\limits_{i\in[1,n_{R}]}\left[\JacR^{-1}\left(\DxR\LamxR\oneR+\DyR\LamyR\oneR\right)\right]_{i},
\end{split}
\end{equation*}
and we have defined new norms $\tildeML = \uL\Tr\ML\JacL\uL$, $\tildeMR=\uR\Tr\MR\JacR\uR$. Using these bounds in
\eqref{eq:energyLR}, and grouping the terms involving the SAT matrices, we obtain
\ignore{
\begin{align*}
-\uL\Tr\UL\ML\left(\DxL\LamxL\oneL+\DyL\LamyL\oneL\right) &\leq \gamma_{L}\|\uL\|_{\ML}^{2}\leq\tilde{\gamma}_{L}\|\uL\|_{\tildeML}^{2},\\
-\uR\Tr\UR\MR\left(\DxR\LamxR\oneR+\DyR\LamyR\oneR\right) &\leq \gamma_{R}\|\uR\|_{\MR}^{2}\leq\tilde{\gamma}_{R}\|\uR\|_{\tildeMR}^{2}
\end{align*}
\todo[inline]{David: do we need to make any assumption on the curvilinear coordinate transformation so that we can do the above?}\todo[inline]{Jason: it might not hurt to state somewhere that we are assuming the coordinate transformation is differentiable and has a bounded and nonzero Jacobian}\todo[inline]{Jason: do we need to define the $\max$ function for a vector?  Can we use the infinity norm instead?  I do not know.}
where $\gamma_{L}=\max\left(\DxL\LamxL\oneL+\DyL\LamyL\oneL\right)$,
$\gamma_{R}=\max\left(\DxR\LamxR\oneR+\DyR\LamyR\oneR\right)$, we have defined
new norms $\tildeML = \uL\Tr\ML\JacL\uL$, $\tildeMR=\uR\Tr\MR\JacR\uR$, and $\tilde{\gamma}_{L}$ and $\tilde{\gamma}_{R}$
are appropriately scaled constants such that the inequality holds.\todo[inline]{Can you do this?  Is there a paper we can cite that explains this?}  Using these bounds in
\eqref{eq:energyLR}, and grouping the terms involving the SAT matrices, we obtain
}
\begin{multline}\label{eq:energy_bound}
\frac{\mr{d}}{\mr{d}t} \left(\|\uL\|_{\tildeML}^{2} + \|\uR\|_{\tildeMR}^{2} \right)
\leq
C\left(\|\uL\|_{\tildeML}^{2}+\|\uR\|_{\tildeMR}^{2}\right)\\
-
\begin{bmatrix} \uL\Tr & \uR\Tr \end{bmatrix}
\begin{bmatrix}
\RL\Tr\BL\RL\LamL-\MLL & \MLR \\
\MRL & \RR\Tr\BR\RR\LamR-\MRR
\end{bmatrix}
\begin{bmatrix} \uL \\ \uR \end{bmatrix},
\end{multline}
for $C = \max(\gamma_{L},\gamma_{R})$.

In order to bound the solution in terms of the initial and boundary conditions, the matrix in the right-hand side must be must be 
negative semi-definite. This motivates the first condition on the SAT matrices. 
\ignore{
As it stands, the bound \eqref{eq:energy_bound} permits arbitrary nonphysical
growth since the SAT terms are not bounded.  To see this more clearly, consider the divergence-free case, in which the terms involving $\UL$ and $\UR$ are absent from the discrete equations \eqref{eq:fullL} and \eqref{eq:fullR}, respectively. In this case it is easy to show that the rate-of-change of discrete energy is \emph{equal} to the terms on the second line of \eqref{eq:energy_bound}.  This motivates the first condition on the SAT matrices.}

\begin{cond}[Stability]\label{eq:conds}
  The matrices $\MLL$, $\MLR$, $\MRR$, and $\MRL$ must be such that
  \begin{equation*}
    \begin{bmatrix}
      \RL\Tr\BL\RL\LamL-\MLL & \MLR \\
      \MRL & \RR\Tr\BR\RR\LamR-\MRR
    \end{bmatrix}
  \end{equation*}
  is positive semi-definite for all $\LamL$ and $\LamR$.
\end{cond}

%
%
\subsubsection{Accuracy}
In order to maintain the accuracy of the base scheme, the SATs must add terms that are of
the order of the discretization. The required conditions are therefore given by

\begin{cond}[Accuracy]\label{eq:conda}
The matrices $\MLL$, $\MLR$, $\MRR$, and $\MRL$ must be such that
\begin{equation*}
\begin{aligned}
\ML^{-1}\left(\MLL\bm{v}_{L}-\MLR\bm{v}_{R}\right) &= \mathcal{O}(h^{\tilde{p}}),\\
\MR^{-1}\left(\MRR\bm{v}_{R}-\MRL\bm{v}_{L}\right) &= \mathcal{O}(h^{\tilde{p}}),
\end{aligned}
\end{equation*}
where $h$ is an appropriate measure for the linear dimension of the shared face,
and $\vL$ and $\vR$ are the projection, onto the nodes of the left and right
domains in physical space, of some continuous function $\fnc{V}(x,y)$, and $\tilde{p}\ge p$.
\end{cond}

We will have more to say about the above accuracy condition in the context of specific
examples of SATs in Section~\ref{sec:concrete}.

%
%
\subsubsection{Conservation}

In order to determine the constraints on the SAT matrices such that the scheme
is conservative, we multiply \eqref{eq:fullL} by $\oneL\Tr\ML$, and
\eqref{eq:fullR} by $\oneR\Tr\MR$ and sum the expressions; this operation is the discrete analogue of
integrating the PDE over the volume consisting of both elements, \ie,
$\tilde{\Omega}$ in \eqref{eq:con}. Simplifying the result we obtain
\begin{multline*}
  \frac{\mr{d}}{\mr{d}t}\left( \oneL\Tr\ML\JacL\uL +
\oneR\Tr\MR\JacR\uR\right)=\\
\frac{1}{2}\left[
-\oneL\Tr\RL\BL\RL\LamL\uL-\oneL\Tr\LamL\RL\Tr\BL\RL
+\oneL\Tr\MLL-\oneR\Tr\MRL\right]\uL,\\
+\frac{1}{2}\left[
-\oneR\Tr\RR\Tr\BR\RR\LamR\uR-\oneR\Tr\LamR\RR\Tr\BR\RR
+\oneR\Tr\MRR-\oneL\Tr\MLR\right]\uR.
\end{multline*}
For conservation, we want the right-hand side to vanish for arbitrary $\uL$ and
$\uR$.  Thus, after rearranging the right-hand side, we arrive at the third, and final, condition on the SAT matrices.
\begin{cond}[Conservation]\label{eq:condc}
  The matrices $\MLL$, $\MLR$, $\MRR$, and $\MRL$ must be such that
  \begin{multline*}
    \left[\oneL\Tr\left(\MLL - \RL\Tr\BL\RL\LamL\right) - \oneR\Tr\MRL\right] \uL \\
    - \left[\oneR\Tr\left(-\MRR + \RR\Tr\BR\RR\LamR\right) + \oneL\Tr\MLR\right]\uR \\
    =\left(\RL\LamL\oneL\right)\Tr\BL\RL\uL+\left(\RR\LamR\oneR\right)\Tr\BR\RR\uR
  \end{multline*}
  for all $\uL$, $\uR$, $\LamL$, $\LamR$.
\end{cond}

In the above condition, we have left the conservation conditions coupled and
dependent on the solution as this is the typical situation for numerical fluxes
used for nonlinear problems.  We elaborate further on the conservation condition
below.

%
%
\subsection{Divergence-free advection field}

We conclude Section~\ref{sec:var} by investigating the special case of a
divergence-free advection field in the variable-coefficient problem
\eqref{eq:LCvar}, \ie, $\nabla \cdot \bm{\lambda} = 0$. This case is of practical
interest, because it arises in the incompressible Navier-Stokes and Maxwell's
equations.  We also use this case to verify the theory for scalar
variable-coefficient equations in Section~\ref{sec:results}.

In the divergence-free case, the skew-symmetric form of the PDE \eqref{eq:LCvar} simplifies to
\begin{equation}\label{eq:divfree}
\frac{\partial\fnc{U}}{\partial t} =
-\frac{1}{2}\vecnabla \cdot\left(\veclambda \fnc{U}\right)
-\frac{1}{2}\veclambda \cdot\nabla \fnc{U},
\end{equation}
and the SBP-SAT semi-discretization of \eqref{eq:divfree} on the left element becomes
\begin{multline}\label{eq:discdivfree}
\frac{\mr{d}}{\mr{d}t} \left( \JacL \uL \right)
+\frac{1}{2}\DxL\LamxL\uL
+\frac{1}{2}\DyL\LamyL\uL
+\frac{1}{2}\LamxL\DxL\uL
+\frac{1}{2}\LamyL\DyL\uL\\
= \frac{1}{2}\ML^{-1}\left(\MLL\uL-\MLR\uR\right).
\end{multline}
The semi-discretization on the right element is similar.

It is straightforward to show that the stability and accuracy conditions remain
the same in the divergence-free case.  However, the conservation condition is
modified:

\begin{customcond}{3$'$}[Divergence-Free Conservation]\label{eq:condc_divfree}
  For the SBP-SAT semi-discretization of \eqref{eq:divfree}, the matrices
  $\MLL$, $\MLR$, $\MRR$, and $\MRL$ must be such that
  \begin{multline*}
    \left[\oneL\Tr\left(\MLL - \RL\Tr\BL\RL\LamL\right) - \oneR\Tr\MRL\right] \uL \\
    - \left[\oneR\Tr\left(-\MRR + \RR\Tr\BR\RR\LamR\right) + \oneL\Tr\MLR\right]\uR = \\
    \oneL\Tr\left(\LamxL \QxL + \LamyL\QyL\right)\uL
    + \oneR\Tr\left(\LamxR \QxR + \LamyL\QyR\right)\uR
  \end{multline*}
  for all $\uL$, $\uR$, $\LamxL$, $\LamyL$, $\LamxR$, and $\LamyR$.
\end{customcond}

Note that the right-hand side of the divergence-free conservation condition involves $\Qx$ and $\Qy$.
\ignore{
\todo[inline]{David: what if we do the following}
and the SBP-SAT semi-discretization of \eqref{eq:divfree} on the $e$th element becomes
\begin{multline*}
\frac{\mr{d}}{\mr{d}t} \left( \Jace^{-1} \ue \right)
+\frac{1}{2}\Dxie\Lamxie\ue
+\frac{1}{2}\Detae\Lametae\ue
+\frac{1}{2}\Lamxie\Dxie\ue
+\frac{1}{2}\Lametae\Detae\ue\\
= \frac{1}{2}\Me^{-1}\sum\limits_{i=1}^{\kappa_{e}}\left(\Meei\ue-\Mei\ui\right).
\end{multline*}

It is straightforward to show that the stability and accuracy conditions remain
the same in the divergence-free case.  However, the conservation condition is
modified:

\begin{customcond}{3$'$}[Divergence-Free Conservation]\label{eq:condc_divfree}
  For the SBP-SAT semi-discretization of \eqref{eq:divfree}, the matrices
  $\Meei$ and $\Mei$ must be such that
  \begin{multline*}
\sum\limits_{e=1}^{E}\sum\limits_{i=1}^{\kappa_{e}}
\onee\Tr\left(\Meei-\Rei\Tr\Be\Rei\Lame-\Lame\Rei\Tr\Be\Rei\right)\ue
-\onee\Tr\Mei\ui
+\onee\Tr\left(\Lamxie\Qxie+\Lametae\Qetae\right)\ue = 0
  \end{multline*}
  for all $\ue$ and $\Lame$.
\end{customcond}
}
%
%
\section{Concrete examples of SATs: symmetric and upwind SATs}\label{sec:concrete}

The SATs presented in Section~\ref{sec:var} offer significant generality, but
they are also somewhat abstract.  In this section, we present two concrete
examples of SATs --- symmetric and upwind --- for multidimensional SBP
discretizations, and we show that these SATs satisfy Conditions 1--3 for
stability, conservation, and accuracy.

The proposed symmetric and upwind SATs require the following assumption that
constrains the face-cubature rules and coordinate transformations of adjacent
elements.

\begin{assume}\label{assume:coord}
  The cubature rule of the face shared by adjacent elements has the same number
  of nodes, $\nu$, in both reference spaces.  In addition, the coordinate
  transformations in the adjacent elements continuously map their respective
  face-cubature nodes to the same locations in physical space.  For example, in
  the case of the left and right elements we have
  \begin{equation*}
    x\left(\xi_{L,i},\eta_{L,i}\right) = x\left(\xi_{R,i},\eta_{R,i}\right)
    \quad\text{and}\quad
    y\left(\xi_{L,i},\eta_{L,i}\right) = y\left(\xi_{R,i},\eta_{R,i}\right),
    \qquad\forall\; i=1,2,\ldots,\nu,
  \end{equation*}
  where $\left(\xi_{L,i},\eta_{L,i}\right)$ and
  $\left(\xi_{R,i},\eta_{R,i}\right)$ denote the $i$th face-cubature points on
  the left and right elements, respectively. Furthermore, the scaled face-normal
  vectors, based on the coordinate transformations along the shared face, are
  equal and opposite at the cubature nodes:
  \begin{equation}\label{eq:normal}
    b_{L,i} \left[\fnc{J} \left( n_{\xi} \nabla \xi + n_{\eta} \nabla \eta \right)\right]_{L,i} = -b_{R,i} \left[\fnc{J}\left( n_{\xi} \nabla \xi  + n_{\eta} \nabla \eta \right) \right]_{R,i},
    \qquad\forall\; i=1,2,\ldots,\nu,
  \end{equation}
  where $b_{L,i}$ and $b_{R,i}$ denote the $i$th face-cubature weights on the
  left and right elements, respectively.
\end{assume}

\begin{remark}
In principle, the cubature rules for the shared face could use a different number of nodes on the left and right elements, but this case is beyond the scope of the current work.
\end{remark}

\begin{remark}
  Equation~\eqref{eq:normal} is satisfied by isoparametric and subparametric
  coordinate transformations that use the same (possibly scaled) cubature rule
  on each face.
\end{remark}

Let $\lambda_{n} = \lambda_{\xi} n_{\xi} + \lambda_{\eta} n_{\eta}$ be the advection velocity normal to the shared face.  Then, assuming that $\lambda_{x}$ and $\lambda_{y}$ are continuous, one can use \eqref{eq:normal} and the definitions of $\lamxi$ and $\lameta$ to show that
\begin{equation}\label{eq:lamLlamR}
  b_{L,i} \left(\lambda_{n}\right)_{L,i}
  = -b_{R,i} \left(\lambda_{n}\right)_{R,i},
  \qquad\forall\; i=1,2,\ldots,\nu.
\end{equation}
In other words, the cubature-scaled advection velocity normal to the face is
equal in magnitude and opposite in direction at the coincident nodes along the
face.  We use \eqref{eq:lamLlamR} to define the diagonal $\nu\times\nu$ matrix
\begin{equation}\label{eq:defBlam}
\Blam = \BL \mat{\Lambda}_{\hat{\Gamma},L} = \BR \mat{\Lambda}_{\hat{\Gamma},R}
\end{equation}
where
\begin{equation*}
  \mat{\Lambda}_{\hat{\Gamma},L} =
  \mydiag\left[\left(\lambda_{n}\right)_{L,1}, \ldots,
  \left(\lambda_{n}\right)_{L,\nu}\right],\qquad
  \mat{\Lambda}_{\hat{\Gamma},R} =
  -\mydiag\left[\left(\lambda_{n}\right)_{R,1}, \ldots,
  \left(\lambda_{n}\right)_{R,\nu}\right].
\end{equation*}
$\Blam$ will play a central role in defining the symmetric and upwind SATs below.

\begin{remark}
  Using \eqref{eq:lamLlamR} to define $\Blam$ simplifies the proof of accuracy
  for the symmetric and upwind SATs, but it is important to emphasize that
  neither \eqref{eq:lamLlamR} nor \eqref{eq:normal} are necessary for stability,
  accuracy, or conservation.  Conditions 1--3 will still hold as long as $\Blam$
  agrees with \eqref{eq:lamLlamR} on the order of the discretization and
  satisfies \eqref{eq:condcsym} below.  This is important, because nonlinear
  problems will not satisfy \eqref{eq:lamLlamR} due to jumps in $\lambda_{x}$
  and $\lambda_{y}$ across elements.
\end{remark}

\ignore{
Both the symmetric and upwind SATs rely on the following $\nu\times \nu$
diagonal matrix defined at the common face nodes:
\begin{equation}\label{eq:defBlam}
\Blam = \frac{1}{2}\left(\BL \mat{\Lambda}_{\hat{\Gamma},L}
+ \BR \mat{\Lambda}_{\hat{\Gamma},R}\right),
\end{equation}
where
\begin{equation*}
  \mat{\Lambda}_{\hat{\Gamma},L} =
  \mydiag\left[\left(\lambda_{n}\right)_{L,1}, \ldots,
  \left(\lambda_{n}\right)_{L,\nu}\right],\qquad
  \mat{\Lambda}_{\hat{\Gamma},R} =
  -\mydiag\left[\left(\lambda_{n}\right)_{R,1}, \ldots,
  \left(\lambda_{n}\right)_{R,\nu}\right].
\end{equation*}
and $\lambda_{n} = \lambda_{\xi} n_{\xi} + \lambda_{\eta} n_{\eta}$.

\begin{remark}
  The definition of $\Blam$ is not unique.  Other choices include $\Blam = \BL
  \mat{\Lambda}_{\hat{\Gamma},L}$ and $\Blam = \BR
  \mat{\Lambda}_{\hat{\Gamma},R}$.  Note, that for isoparametric and
  subparametric coordinate transformations, we have
  \begin{equation*}
    \BL\mat{\Lambda}_{\hat{\Gamma},L} = \BR\mat{\Lambda}_{\hat{\Gamma},R},
  \end{equation*}
  because the scaled face-normal vector is continuous across the interface for
  such transformations (since $\lambda_{x}$ and $\lambda_{y}$ are continuous).
  However, this continuity is not required here, which is important for the
  generalization to nonlinear PDEs.
\end{remark}
}

\subsection{Symmetric SATs}

Symmetric SATs are constructed by defining
\begin{equation}\label{eq:symM}
\begin{alignedat}{2}
\MLL &= \RL\Tr\BL\RL\LamL,\qquad & \MLR &= \RL\Tr\Blam\RR, \\
\MRR &= \RR\Tr\BR\RR\LamR,\qquad & \MRL &= -\RR\Tr\Blam\RL.
\end{alignedat}
\end{equation}
Based on these matrices, symmetric SATs for \eqref{eq:fullL} and
\eqref{eq:fullR} are given by
\begin{equation}\label{eq:SATsym}
\begin{aligned}
2\ML\bm{\mr{SAT}}_{L,\mr{sym}} &= \RL\Tr\BL\RL\LamL\uL-\RL\Tr\Blam\RR\uR,\\
2\MR\bm{\mr{SAT}}_{R,\mr{sym}} &= \RR\Tr\BR\RR\LamR\uR+\RR\Tr\Blam\RL\uL,
\end{aligned}
\end{equation}

\begin{thrm}
  The symmetric SATs \eqref{eq:SATsym} satisfy the stability and accuracy
  Conditions \ref{eq:conds} and \ref{eq:conda}.  In addition, they satisfy the
  conservation Condition \ref{eq:condc} provided
  \begin{equation}\label{eq:condcsym}
    \oneG\Tr\Blam\left(\RL\uL-\RR\uR\right) = \left(\RL\LamL\oneL\right)\Tr\BL\RL\uL+\left(\RR\LamR\oneR\right)\Tr\BR\RR\uR
  \end{equation}
  for all $\uL$, $\uR$, $\LamL$, $\LamR$, where $\oneG$ is a vector of ones of
  length $\nu$.
\end{thrm}

\begin{proof}
  It is easy to see that the symmetric SAT matrices \eqref{eq:symM} lead to a
  skew-symmetric matrix in Condition \ref{eq:conds}, which implies that the SATs
  \eqref{eq:SATsym} are stable.

  To prove that the symmetric SATs satisfy the accuracy condition, we show that
  they vanish for all polynomial face-normal fluxes, $\left(\lambda_{\xi}n_{\xi}
  + \lambda_{\eta} n_{\eta}\right)\fnc{U}$, of total degree $p$ or less on the
  reference domain.  We do this for the left-element SAT only, since the proof
  is analogous for the right-element SAT.

  Let $\LamL\uL \equiv \pkL$ be the face-normal polynomial flux evaluated at the SBP nodes of the left element (in reference space) 
--- where we consider all $\LamL$ and $\uL$ that satisfy this definition --- and let $\pkgammaL$ be the same polynomial evaluated at the face-cubature points on the left element.
  Then, we consider those states on the right element such that
  \begin{equation*}
    \mat{\Lambda}_{\hat{\Gamma},L} (\RR\uR) = \pkgammaL.
  \end{equation*}
  Note that such states $\uR$ exist provided $\mat{\Lambda}_{\hat{\Gamma},L} \RR$ is full rank.
The vector $\pkgammaL$ defines the ``boundary'' flux for which we must show the left SAT vanishes.  We have
  \begin{align*}
    \RL\Tr\BL\RL \LamL\uL - \RL\Tr\Blam\RR\uR
    &= \RL\Tr\BL\RL \LamL\uL - \RL\Tr\BL\mat{\Lambda}_{\hat{\Gamma},L}\RR\uR \\
    &= \RL\Tr\BL \left(\RL \pkL - \pkgammaL\right) = 0,
    \qquad \forall\; k \in \{ 1,2,\ldots,\nmin{p} \},
  \end{align*}
  where we have used $\Blam = \BL\mat{\Lambda}_{\hat{\Gamma},L}$ and the fact that $\RL$ is exact for polynomials of degree $r$ or less, where $r\geq p$.  Thus, the SAT is zero for all polynomial face-normal fluxes of total degree $p$ or less.

  Finally, we substitute the $\mat{M}$ matrices into the equation in
  Condition~\ref{eq:condc} and find
  \begin{equation*}
    \oneG\Tr\Blam\left(\RL\uL-\RR\uR\right) = \left(\RL\LamL\oneL\right)\Tr\BL\RL\uL+\left(\RR\LamR\oneR\right)\Tr\BR\RR\uR,
  \end{equation*}
  which is precisely \eqref{eq:condcsym}.  Therefore, if this constraint is
  satisfied, the symmetric SATs are conservative.\qed
\end{proof}

Equation~\eqref{eq:condcsym} can be viewed as a constraint on the variable
coefficient matrices $\LamL$, $\LamR$, and $\Blam$.  There are a few ways this
constraint can be satisfied:
\begin{itemize}
\item For scalar variable-coefficient advection, we can apply a preprocessing
  step to enforce discrete continuity of the face-normal velocities, that is
  $\BL\RL\LamL\oneL = -\BR\RR\LamR\oneR = \Blam\oneG$; we use a similar
  preprocessing step for the divergence-free variable-coefficient advection case
  presented in the results.

\item For nonlinear systems of PDEs, such as the Euler equations of gas
  dynamics, the variable coefficients are functions of the solution and the
  coordinate transformation.  In this case, it is more convenient to consider
  pointwise conditions on a numerical flux Jacobian.  To illustrate, in the
  scalar case we would have
  \begin{equation*}
    \bar{\lambda}_{i}\left[\left(\RL\uL\right)_{i} - \left(\RR\uR\right)_{i}\right]
    = F\left[\left(\RL\uL\right)_{i},\bm{n}_{i,L}\right]
      + F\left[\left(\RR\uR\right)_{i},\bm{n}_{i,R}\right],
    \quad\forall\; i=1,2,\dots,\nu,
  \end{equation*}
  where $F[u,\bm{n}]$ is the nonlinear flux in the direction $\bm{n}$, and
  $\bar{\lambda}_{i}$ is the numerical flux Jacobian at the $i$th common face
  node.  Note that the pointwise conditions define a secant-like equation for
  the numerical flux Jacobian, which is a common condition for numerical fluxes
  used in nonlinear hyperbolic systems.

\end{itemize}

\subsection{Upwind SATs}\label{sec:upwind}

To construct upwind SATs, we define
\begin{equation}\label{eq:upwindM}
\begin{alignedat}{2}
\MLL &=\RL\Tr\BL\RL\LamL - \RL\Tr|\Blam|\RL,\qquad&
\MLR &= \RL\Tr\left(\Blam-|\Blam|\right)\RR,\\
\MRR &= \RR\Tr\BR\RR\LamR - \RR\Tr|\Blam|\RR, \qquad&
\MRL &= -\RR\Tr\left(\Blam + |\Blam|\right)\RL.
\end{alignedat}
\end{equation}
Therefore, upwind SATs for \eqref{eq:fullL} and \eqref{eq:fullR} are given by
\begin{equation}\label{eq:SATupwind}
\begin{aligned}
2\ML\bm{\mr{SAT}}_{L,\mr{upwd}} &= \left(\RL\Tr\BL\RL\LamL - \RL\Tr|\Blam|\RL\right)\uL
 - \RL\Tr\left(\Blam-|\Blam|\right)\RR\uR,\\
2\MR\bm{\mr{SAT}}_{R,\mr{upwd}} &= \left(\RR\Tr\BR\RR\LamR - \RR\Tr|\Blam|\RR\right)\uR
+ \RR\Tr\left(\Blam + |\Blam|\right)\RL\uL.
\end{aligned}
\end{equation}

\ignore{
\todo[inline]{Jason: I deleted this, because a penalty is a difference between the state and the desired state, not these terms}
We can see that the terms $\RL\Tr\left(\Blam-|\Blam|\right)\RR\uR$ and
$\RR\Tr\left(\Blam + |\Blam|\right)\RL\uL$ act as upwind penalty terms. The
remaining terms introduce small errors proportional to the difference in the
projection of the solution multiplied by the variable coefficient and the
projection of the solution multiplied by the variable coefficient at the
auxiliary nodes.
}

\begin{thrm}
The upwind SATs \eqref{eq:SATupwind} satisfy the stability and accuracy
  Conditions \ref{eq:conds} and \ref{eq:conda}.  In addition, they satisfy the
  conservation Condition \ref{eq:condc} provided
  \begin{equation}\tag{\ref{eq:condcsym}}
    \oneG\Tr\Blam\left(\RL\uL-\RR\uR\right) = \left(\RL\LamL\oneL\right)\Tr\BL\RL\uL+\left(\RR\LamR\oneR\right)\Tr\BR\RR\uR
  \end{equation}
  for all $\uL$, $\uR$, $\LamL$, $\LamR$, where $\oneG$ is a vector of ones of
  length $\nu$.
\end{thrm}

\begin{proof}
  The matrix in Condition \eqref{eq:conds} is positive semi-definite using the
  upwind SAT matrices \eqref{eq:upwindM} if
\begin{equation*}
  \begin{bmatrix} (\RL\uL)\Tr & (\RR\uR)\Tr \end{bmatrix}
  \begin{bmatrix} \phantom{-}|\Blam| & -|\Blam| \\
    -|\Blam| & \phantom{-}|\Blam| \end{bmatrix}
 \begin{bmatrix} \RL\uL \\ \RR\uR \end{bmatrix} \geq 0,
\end{equation*}
for all nonzeros $\uL$ and $\uR$.  This is satisfied, because the matrix
$\left[\begin{smallmatrix} \phantom{-}|\Blam| & -|\Blam| \\ -|\Blam| &
    \phantom{-}|\Blam| \end{smallmatrix} \right]$ has non-negative eigenvalues.

The proof that the upwind SATs satisfy the accuracy Condition \ref{eq:conda} is
similar to the accuracy proof of the symmetric SATs, so we omit it for
brevity.

Substituting the upwind SAT matrices \eqref{eq:upwindM} into the conservation
condition, Condition~\ref{eq:condc}, we obtain the same constraint on the
variable coefficients as for the symmetric SATs,
namely~\eqref{eq:condcsym}.\qed
\end{proof}

\subsection{Divergence-free advection field with upwind SATs}\label{sec:divfree}

We consider the use of the upwind SATs in the SBP-SAT discretization of the
divergence-free variable-coefficient problem, \eqref{eq:divfree}, because this is the
PDE and the SATs we employ in the results presented below.  As remarked
previously, the divergence-free case does not alter the stability or accuracy of
the discretization.  Thus, we need only address the conservation condition.

With upwind SATs, the equation in Condition~\ref{eq:condc_divfree} reduces to
\begin{equation}\label{eq:condc_upwind}
  \oneG\Tr\Blam \left( \RL \uL - \RR\uR \right)
  =
  \oneL\Tr\left(\LamxL \QxL + \LamyL\QyL\right)\uL
  + \oneR\Tr\left(\LamxR \QxR + \LamyL\QyR\right)\uR,
\end{equation}
where we have made use of $\RL\oneL = \RR\oneR = \oneG$.  Unlike the
non-divergence-free situation, the conservation condition
\eqref{eq:condc_upwind} is no longer local to the common face.

One way to satisfy conservation in this case is to define the discrete
divergence-free condition in such a way that \eqref{eq:condc_upwind} is
satisfied.  In particular, we require that $\Lamx$, $\Lamy$ and $\B_{\lambda,j}$
satisfy
\begin{equation}\label{eq:divfree_h}
  \left(\Dx \Lamx + \Dy \Lamy\right)\bm{1} = \M^{-1} \sum_{j=1}^{\kappa} \left(\R_{j}\Tr\B_{j} \R_{j} \mat{\Lambda}_{j} - \R_{j}\Tr\B_{\lambda,j} \R_{j}\right) \bm{1},
\end{equation}
on all elements, where $\mat{\Lambda}_{j} = n_{\xi j} \Lamx + n_{\eta j} \Lamy$,
and $\B_{\lambda,j}$ is analogous to $\Blam$ for face $j$.  The left-hand side
of \eqref{eq:divfree_h} is a direct SBP discretization of the divergence-free
condition, while the right-hand side is a SAT-like penalty.  Our approach to
satisfying \eqref{eq:divfree_h} is described in Appendix~\ref{sec:Dlambda}.

If \eqref{eq:divfree_h} is satisfied, it follows from the properties of SBP operators that
\begin{equation}\label{eq:divfree_identity}
  \bm{1}^{T}\left(\Lamx \Qx + \Lamy \Qy\right) \bm{v}
  = \sum_{j=1}^{\kappa} \left(\bm{1}_{\hat{\Gamma}_{j}}\Tr\B_{\lambda,j} \R_{j}\right)\bm{v},\qquad\forall\; \bm{v} \in \mathbb{R}^{n}.
\end{equation}
Using identity \eqref{eq:divfree_identity} in \eqref{eq:condc_upwind} --- and
neglecting SATs on the non-shared faces as usual --- we find that the
conservation condition is satisfied.

\begin{remark}
  Divergence-free equations also arise in the so-called metric invariants that
  are needed for ``free-stream'' preservation; see, for example,
  \cite{Thomas1979geometric}.  These metric invariants can also be satisfied
  using the approach described in Appendix~\ref{sec:Dlambda}, by setting
  $\left[\lambda_{x},\lambda_{y}\right]\Tr = [1,0]\Tr$ and $[0,1]\Tr$, in turn.
\end{remark}

%
%

\section{Example operators on the triangle}\label{sec:triSATs}

In this section, we describe the construction of multi-dimensional SBP operators
on triangular elements in conjunction with the matrices $\R$ and $\B$ that
define the SATs.  We present two families of SBP operators for the triangle.
The first family was presented previously in~\cite{multiSBP}.  This family
consists of operators with $p+1$ nodes on each face and will be referred to as
the \SBPGamma family.  Figure~\ref{fig:nodes_bndry} shows the $p=1$ through
$p=4$ operators from this family.  The second family of triangular-element SBP
operators has strictly interior nodes.  This family will be referred to as the
\SBPOmega family, and the first four operators in this family\footnote{We do
  not consider the $p=0$ operator in this work} are shown in
Figure~\ref{fig:nodes_intr}.

The algorithmic steps involved in constructing the operators are listed below.
The process is similar to that outlined in~\cite{multiSBP} for \SBPGamma, with a
few minor changes that are highlighted.
\begin{enumerate}
\item For a given design accuracy $p$, a symmetric cubature rule is selected or
  constructed that is exact for polynomials of total degree $2p-1$ and has at
  least $\nmin{p}$ nodes.  The nodes for the \SBPOmega family are required
  to be strictly interior, and the \SBPGamma family is required to have $p+1$
  nodes on each face, including the vertices.  For all \SBPOmega operators
  considered here ($p=1,\ldots,4$), there are exactly $\nmin{p}$ cubature
  nodes, whereas the \SBPGamma operators generally have more nodes for the same
  value of $p$.

\item A Legendre-Gauss quadrature rule with $p+1$ nodes is used to define $\B_{\nu}$
  on all faces, \ie the same quadrature rule is used for all three sides,
  although this is not strictly necessary.

\item Let $\Gamhat_{j}$ denote one of the faces of the triangle.  Then
  the volume-to-face interpolation/extrapolation operator for this face is defined by $\R =
  \V_{\Gamhat_{j}} (\V_{\Omhat})^{\dagger}$, where
  $\V_{\Gamhat_{j}}$ denotes an orthogonal polynomial basis evaluated at
  the quadrature nodes of $\Gamhat_{j}$, and the superscript $\dagger$
  denotes the Moore-Penrose pseudoinverse.  The definition of
  $\V_{\Omhat}$ depends on whether we are constructing the \SBPGamma or
  \SBPOmega family.  For the latter, $\V_{\Omhat}$ is an orthogonal
  polynomial basis evaluated at all of the nodes in the volume.  In
  contrast, for the \SBPGamma family, $\V_{\Omhat}$ is the basis evaluated
  at the $p+1$ volume nodes that lie on face $\Gamhat_{j}$.

  Although we have considered only the face $\Gamhat_{j}$, symmetry
  allows the same $\R$ matrix to be used on all three faces simply by permuting
  indices of the volume nodes.

\item The boundary operator $\Ex$ is constructed from the face cubature $\B$ and
  interpolation operator $\R$ using~\eqref{eq:SATdecomp}.  An analogous
  equation is used for $\Ey$.

\item The skew-symmetric operators $\Sx$ and $\Sy$ are determined using the
  accuracy conditions, Property~\ref{sbp:accuracy} of Definition~\ref{def:SBP}.
  For the \SBPOmega operators considered here, the $\Sx$ and $\Sy$ operators
  are fully determined by the accuracy conditions; in contrast, the \SBPGamma
  operators are underdetermined by the accuracy conditions, so the minimum-norm
  solution is used for those operators.
\end{enumerate}

\ignore{ In the final paper, we may wish to do a theoretical complexity study
  for D-SBP implementations}

\begin{figure}[t]
  \renewcommand{\thesubfigure}{}
  \subfigure[$p=1$ \label{fig:p1_bndry}]{%
    \includegraphics[width=0.24\textwidth]{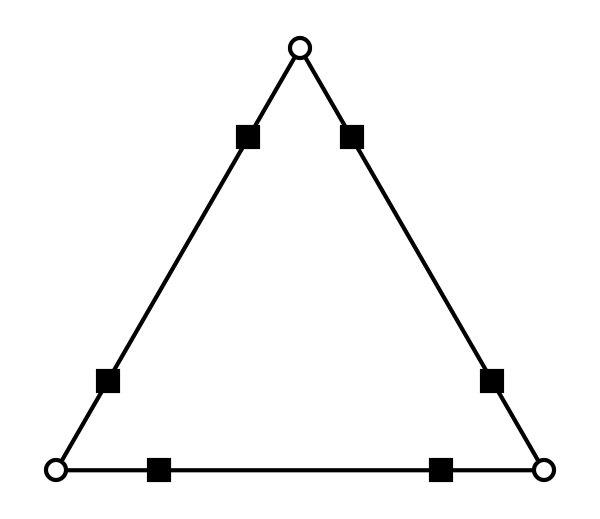}}
  \subfigure[$p=2$ \label{fig:p2_bndry}]{%
        \includegraphics[width=0.24\textwidth]{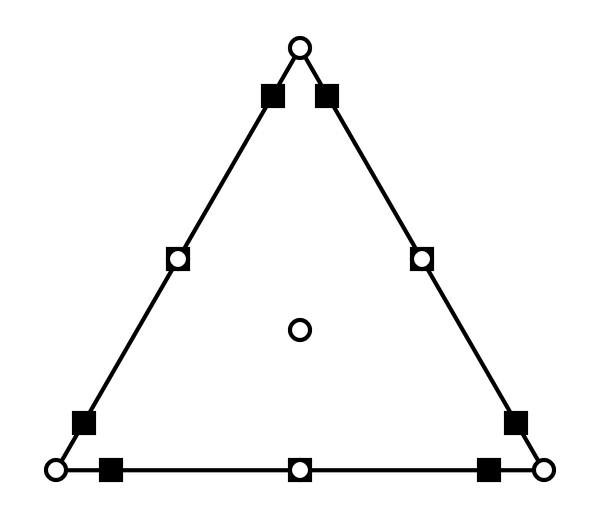}}
  \subfigure[$p=3$ \label{fig:p3_bndry}]{%
        \includegraphics[width=0.24\textwidth]{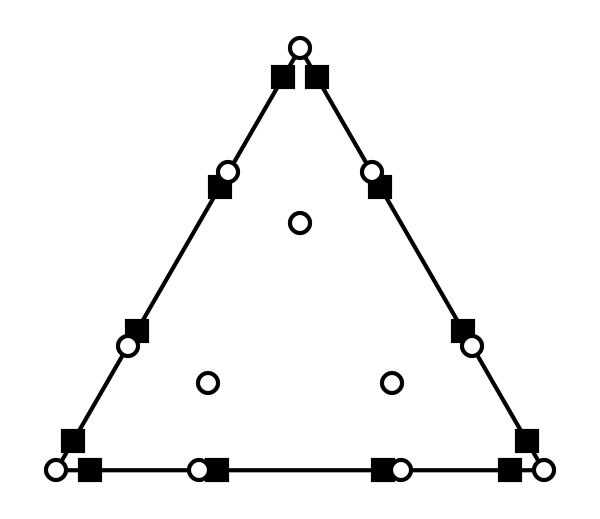}}
  \subfigure[$p=4$ \label{fig:p4_bndry}]{%
        \includegraphics[width=0.24\textwidth]{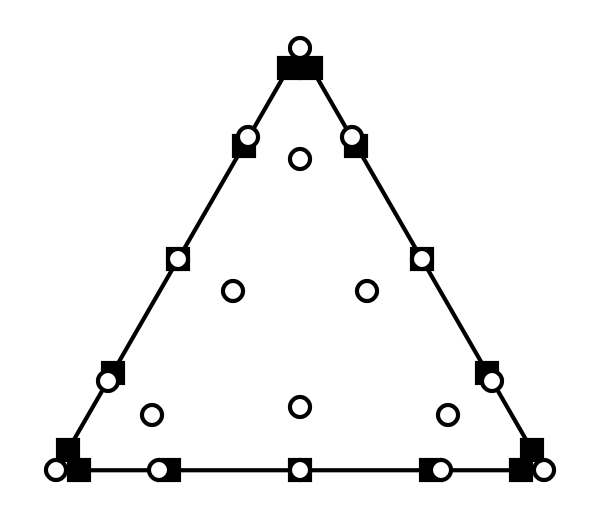}}
  \caption{Nodes of the \SBPGamma family of operators that include $p+1$ nodes
    on each face.  The open circles denote the SBP operator nodes, while the black squares denote the face cubature points used for the SATs.\label{fig:nodes_bndry}}
\end{figure}

\begin{figure}[t]
  \renewcommand{\thesubfigure}{}
  \subfigure[$p=1$ \label{fig:p1_intr}]{%
    \includegraphics[width=0.24\textwidth]{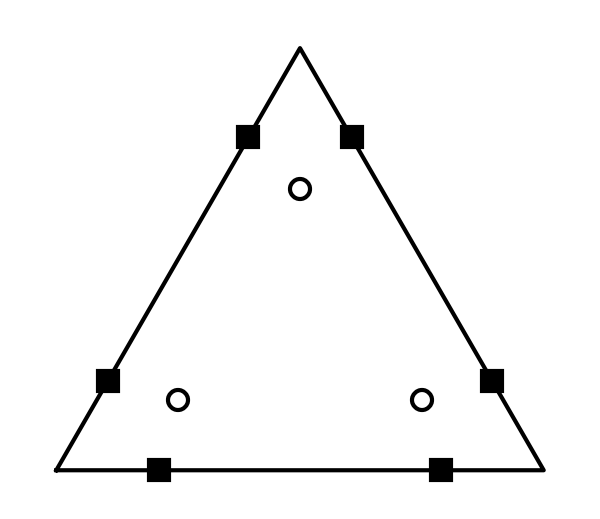}}
  \subfigure[$p=2$ \label{fig:p2_intr}]{%
        \includegraphics[width=0.24\textwidth]{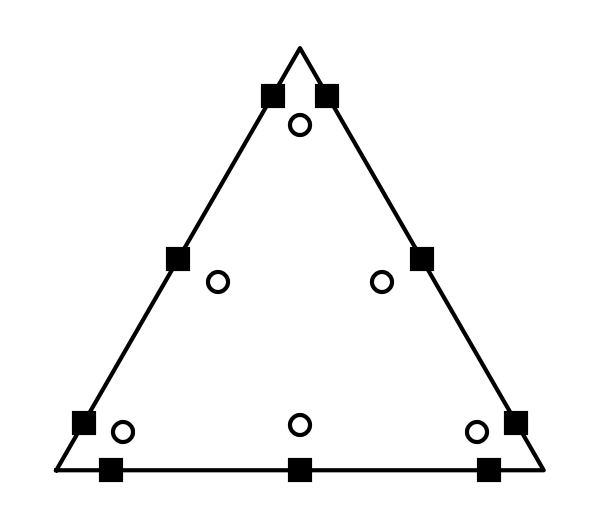}}
  \subfigure[$p=3$ \label{fig:p3_intr}]{%
        \includegraphics[width=0.24\textwidth]{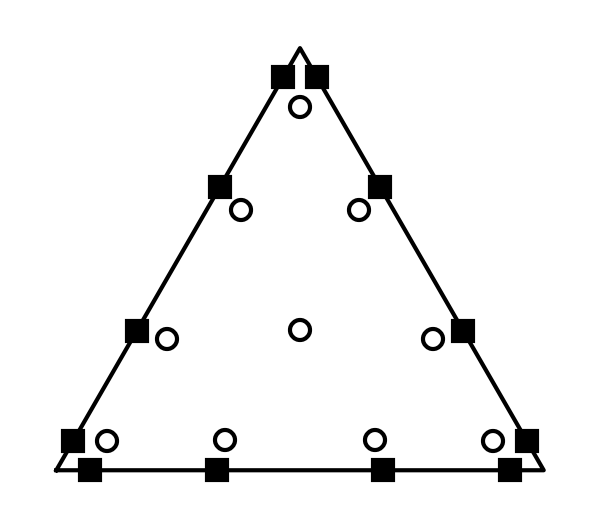}}
  \subfigure[$p=4$ \label{fig:p4_intr}]{%
        \includegraphics[width=0.24\textwidth]{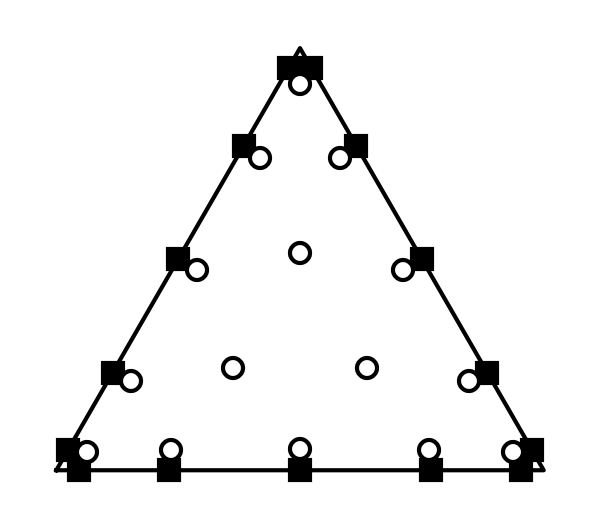}}
  \caption{Nodes of the \SBPOmega family of operators whose nodes are strictly
    interior to $\Omega$. The open circles denote the SBP operator nodes, while
    the black squares denote the face cubature points used for the
    SATs.\label{fig:nodes_intr}}
\end{figure}
\ignore{
\begin{remark}
  The explicit use of $\R$ and $\B$ to define the $\Ex$ and $\Ey$ matrices is
  important, because the same $\R$ and $\B$ are subsequently used to define
  SATs.  For the stability and conservation analysis of Section~\ref{sec:scvar}
  to hold, the $\E$ matrices and SATs must be defined in a consistent manner.
\end{remark}
}

Table~\ref{tab:SBPtri} summarizes the accuracy and node-set properties of both
the \SBPOmega and \SBPGamma families.  Beyond the fact that \SBPGamma includes
boundary nodes and \SBPOmega excludes boundary nodes, a few other differences
between the families are worth highlighting.  First, the \SBPGamma family
generally requires more nodes than the \SBPOmega family for the same design
accuracy $p$; this translates into $\Dx$ and $\Dy$ operators that require more
storage and computation, at least for hyperbolic problems.  Second, the cubature
accuracy is higher for the \SBPOmega family; the $p=1$ and $p=2$ operators have
cubatures that are exact to degree $2p$, rather than $2p-1$, and the $p=3$ and
$p=4$ operators have smaller error constants.  Finally, the volume-to-face
interpolation operators used by the \SBPGamma operators have fewer entries,
giving them a computational advantage when it comes to evaluating the SATs.

\begin{table}[t]
  \begin{center}
    \caption[]{Summary of cubature accuracy, node counts, and operator
      dimensions for the two different families of SBP operators on the
      triangle. \label{tab:SBPtri}}
    \begin{tabular}{lcccc}
      \textbf{family} & \textbf{degree} ($p$) & \textbf{\# nodes} ($n$) & $\M$ \textbf{degree} & $\R$ \textbf{matrix size} \\\hline
      \SBPGamma & 1 & 3  & 1 & $2\times 2$\rule{0ex}{3ex} \\
      \SBPOmega & 1 & 3  & 2 & $2\times 3$ \\\hline
      \SBPGamma & 2 & 7  & 3 & $3\times 3$\rule{0ex}{3ex} \\
      \SBPOmega & 2 & 6  & 4 & $3\times 6$ \\\hline
      \SBPGamma & 3 & 12 & 5 & $4\times 4$\rule{0ex}{3ex} \\
      \SBPOmega & 3 & 10 & 5 & $4\times 10$ \\\hline
      \SBPGamma & 4 & 18 & 7 & $5\times 5$\rule{0ex}{3ex} \\
      \SBPOmega & 4 & 15 & 7 & $5\times 15$ \\\hline
    \end{tabular}
  \end{center}
\end{table}
\ignore{
The above analysis is general, since it applies to any multi-dimensional SBP
discretization with SATs that use the decomposition described by
Theorem~\ref{thm:SATdecomp}.  In an effort to make these ideas more concrete, in
this section and the next we consider SBP-SAT discretizations on grids composed
of triangular elements.

As shown in~\cite{multiSBP}, a multi-dimensional SBP operator of degree $p$
exists for any domain that has 1) a cubature rule of degree at least $2p-1$ with
2) strictly positive weights and 3) a full-rank generalized Vandermonde matrix.
In general, such cubature rules are not unique, and there exists an infinite
number of SBP operators for a given domain.  To illustrate this fact, we present
two families of SBP operators for the triangle.

The first family of SBP operators on the triangle was presented previously
in~\cite{multiSBP}.  This family consists of operators with $p+1$ nodes on each
face and will be referred to as the \SBPGamma family.
Figure~\ref{fig:nodes_bndry} shows the $p=1$ through $p=4$ nodal distributions
from this family.  This family is analogous to tensor-product spectral elements
that use Legendre-Gauss-Lobatto quadrature nodes.

For this work, a second family of triangular-element SBP operators were
constructed that has strictly interior nodes.  This family will be referred to
as the \SBPOmega family, and the first four operators in this family\footnote{We
  did not consider the $p=0$ operator in this work} are shown in
Figure~\ref{fig:nodes_intr}.  This family is akin to tensor-product spectral
elements that use Legendre-Gauss quadrature nodes.

The steps involved in constructing the operators are listed below.
The process is similar to that outlined in~\cite{multiSBP} for \SBPGamma, with a
few minor changes that are highlighted.
\begin{enumerate}
\item For a given design accuracy $p$, a symmetric cubature rule is selected (or
  computed) that is exact for polynomials of total degree $2p-1$ and has at
  least $(p+1)(p+2)/2$ nodes.  The nodes for the \SBPOmega family are required
  to be strictly interior, while the \SBPGamma family is required to have $p+1$
  nodes on each face, including the vertices.  For all \SBPOmega operators
  considered here ($p=1,\ldots,4$), there are exactly $(p+1)(p+2)/2$ cubature
  nodes, whereas the \SBPGamma operators generally have more nodes for the same
  value of $p$.

\item A Legendre-Gauss quadrature rule with $p+1$ nodes is used to define $\B$
  on all faces, \ie the same quadrature rule is used for all three sides,
  although this is not strictly necessary.

\item Let $\Gamma_{\nu}$ denote one of the faces of the triangle.  Then the
  volume-to-face interpolation operator for this face is defined by $\R =
  \V_{\Gamma_{\nu}} (\V_{\Omega})^{\dagger}$, where $\V_{\Gamma_{\nu}}$ denotes
  an orthogonal polynomial basis evaluated at the quadrature nodes of
  $\Gamma_{\nu}$ and the superscript $\dagger$ denotes the Moore-Penrose
  pseudoinverse.  The definition of $\V_{\Omega}$ depends on whether we are
  constructing the \SBPGamma or \SBPOmega family.  For the latter, $\V_{\Omega}$
  is the orthogonal polynomial basis evaluated at all of the nodes in the volume
  cubature.  In contrast, for the \SBPGamma family, $\V_{\Omega}$ is the basis
  evaluated at the $p+1$ volume cubature nodes that lie on face $\Gamma_{\nu}$.

  Although $\R$ is constructed explicitly only on face $\Gamma_{\nu}$, symmetry
  allows the same $\R$ matrix to be used on all three faces simply by permuting
  indices of the volume nodes.

\item The boundary operator $\Ex$ is constructed from the face cubature $\B$ and
  interpolation operator $\R$ using equation~\eqref{eq:SATdecomp}.  An analogous
  equation is used for $\Ey$.

\item The skew-symmetric operators $\Qx$ and $\Qy$ are determined using the
  accuracy conditions, Property~\ref{sbp:accuracy} of Definition~\ref{def:SBP}.
  For the \SBPOmega operators considered here, the $\Qx$ and $\Qy$ operators are
  fully determined by the accuracy conditions.  The \SBPGamma operators for
  $p=1$ and $p=2$ are also uniquely defined, whereas the $p=3$ and $p=4$
  operators are underdetermined by the accuracy conditions; for these operators
  the minimum-norm solution is used to define $\Qx$ and $\Qy$.
\end{enumerate}

\begin{figure}[t]
  \renewcommand{\thesubfigure}{}
  \subfigure[$p=1$ \label{fig:p1_bndry}]{%
    \includegraphics[width=0.24\textwidth]{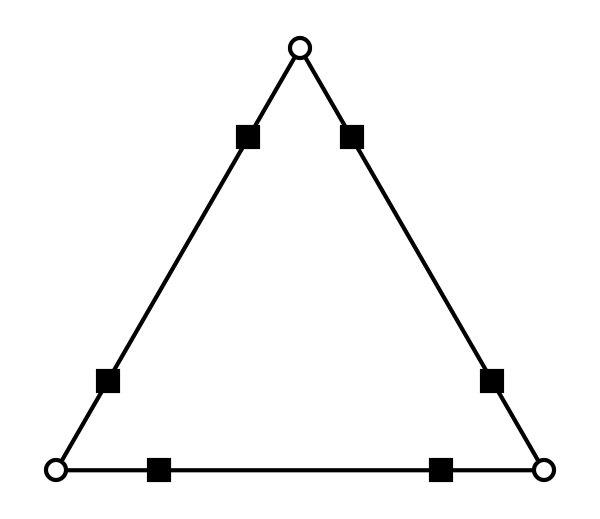}}
  \subfigure[$p=2$ \label{fig:p2_bndry}]{%
        \includegraphics[width=0.24\textwidth]{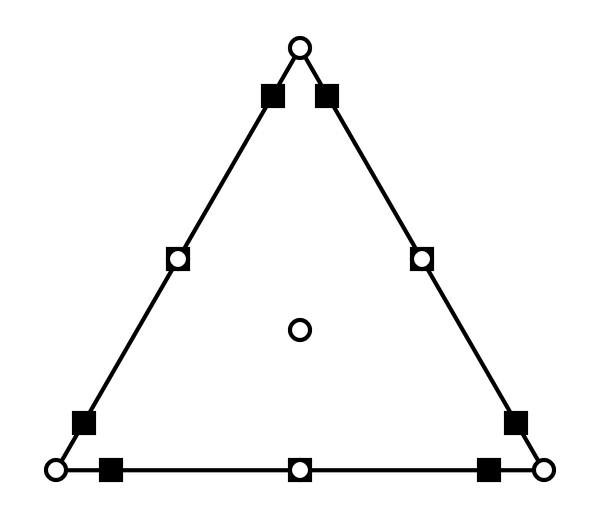}}
  \subfigure[$p=3$ \label{fig:p3_bndry}]{%
        \includegraphics[width=0.24\textwidth]{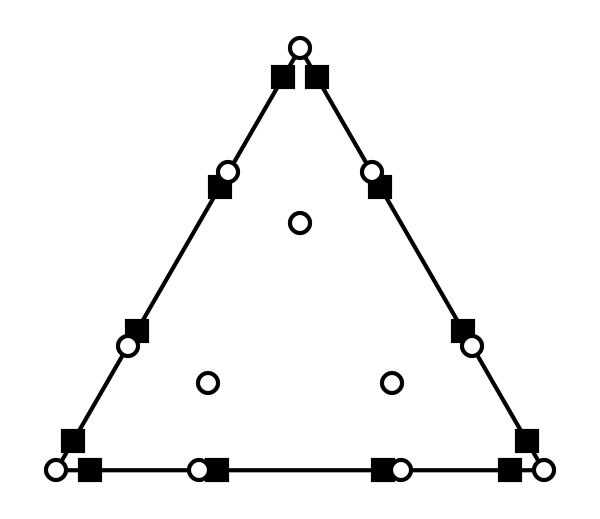}}
  \subfigure[$p=4$ \label{fig:p4_bndry}]{%
        \includegraphics[width=0.24\textwidth]{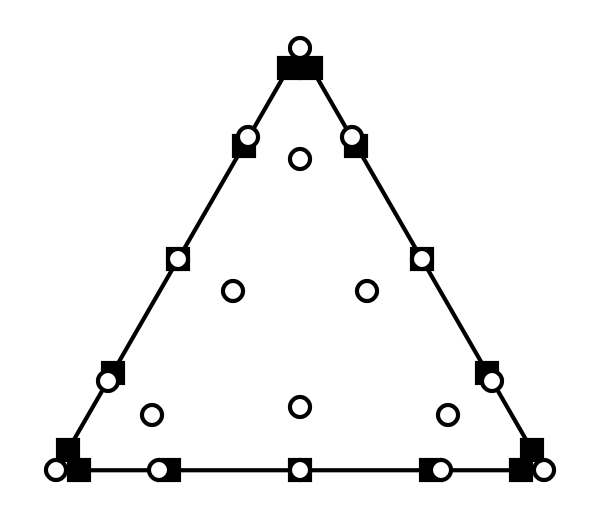}}
  \caption{Nodes of the \SBPGamma family of operators that include $p+1$ nodes
    on each face.  The open circles denote the SBP operator nodes, while the black squares denote the face cubature points used for the SATs.\label{fig:nodes_bndry}}
\end{figure}

\begin{figure}[t]
  \renewcommand{\thesubfigure}{}
  \subfigure[$p=1$ \label{fig:p1_intr}]{%
    \includegraphics[width=0.24\textwidth]{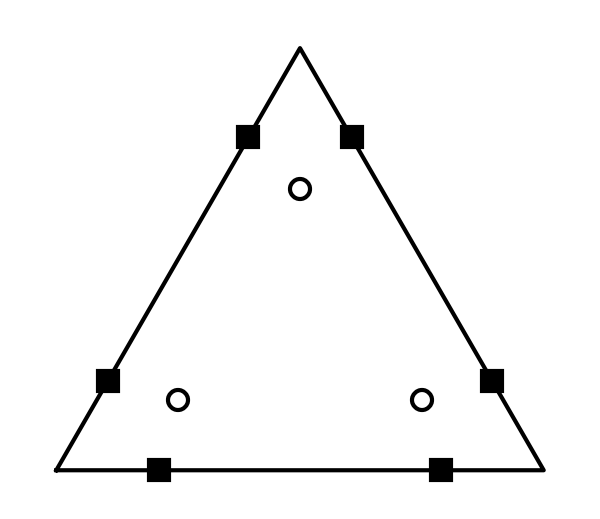}}
  \subfigure[$p=2$ \label{fig:p2_intr}]{%
        \includegraphics[width=0.24\textwidth]{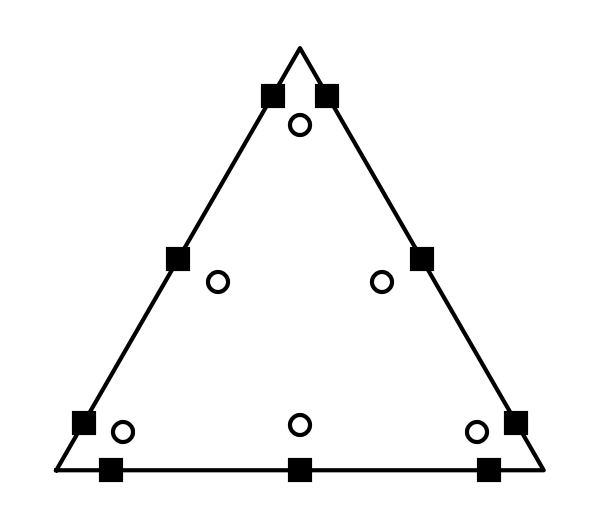}}
  \subfigure[$p=3$ \label{fig:p3_intr}]{%
        \includegraphics[width=0.24\textwidth]{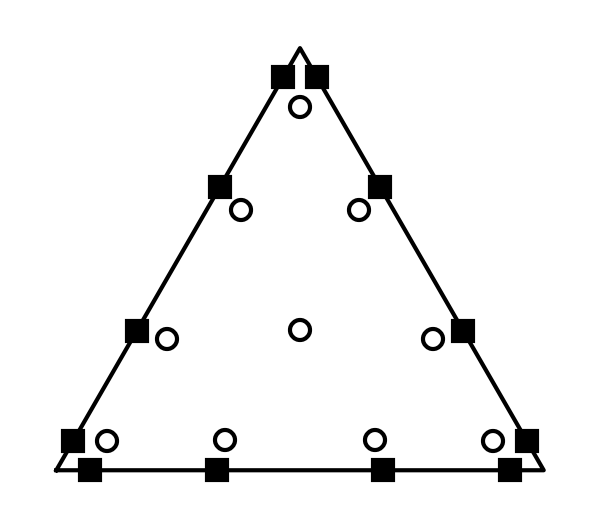}}
  \subfigure[$p=4$ \label{fig:p4_intr}]{%
        \includegraphics[width=0.24\textwidth]{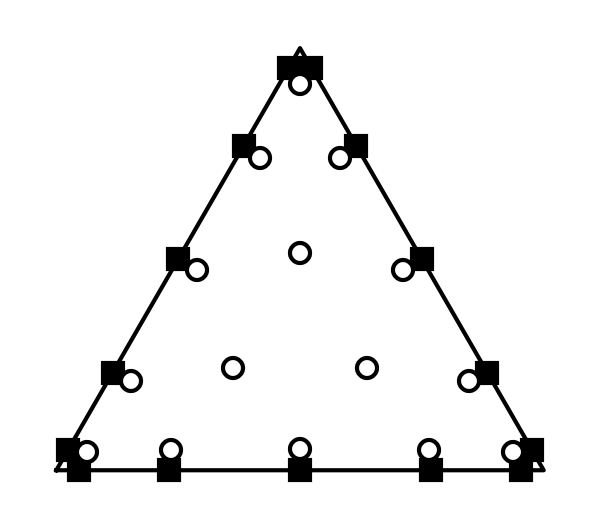}}
  \caption{Nodes of the \SBPOmega family of operators whose nodes are strictly
    interior to $\Omega$. The open circles denote the SBP operator nodes, while
    the black squares denote the face cubature points used for the
    SATs.\label{fig:nodes_intr}}
\end{figure}

\begin{remark}
  The explicit use of $\R$ and $\B$ to define the $\Ex$ and $\Ey$ matrices is
  important, because the same $\R$ and $\B$ are subsequently used to define
  SATs.  For the stability and conservation analysis of Section~\ref{sec:scvar}
  to hold, the $\E$ matrices and SATs must be defined in a consistent manner.
\end{remark}

Table~\ref{tab:SBPtri} summarizes the accuracy and node-set properties of both
the \SBPOmega and \SBPGamma families.  Beyond the fact that \SBPGamma includes
boundary nodes and \SBPOmega excludes boundary nodes, a few other differences
between the families are worth highlighting.  First, the \SBPGamma family
generally requires more nodes than the \SBPOmega family for the same design
accuracy $p$; this translates into $\Dx$ and $\Dy$ operators that require more
storage and computation.  Second, the $p=1$ and $p=2$ operators in the \SBPOmega
family have cubatures that are exact for polynomials of degree $2p$ (rather than
$2p-1$).  Finally, the volume-to-face interpolation operators used by the
\SBPGamma operators are smaller, giving them a computational advantage when it
comes to evaluating the SATs.

\begin{table}[t]
  \begin{center}
    \caption[]{Summary of cubature accuracy, node counts, and operator
      dimensions for the two different families of SBP operators on the
      triangle. \label{tab:SBPtri}}
    \begin{tabular}{lcccc}
      \textbf{family} & \textbf{degree} ($p$) & \textbf{\# nodes} ($n$) & $\M$ \textbf{degree} & $\R$ \textbf{matrix size} \\\hline
      \SBPGamma & 1 & 3  & 1 & $2\times 2$ \rule{0ex}{3ex} \\
      \SBPOmega & 1 & 3  & 2 & $2\times 3$ \\\hline
      \SBPGamma & 2 & 7  & 3 & $3\times 3$ \rule{0ex}{3ex} \\
      \SBPOmega & 2 & 6  & 4 & $3\times 6$ \\\hline
      \SBPGamma & 3 & 12 & 5 & $4\times 4$ \rule{0ex}{3ex} \\
      \SBPOmega & 3 & 10 & 5 & $4\times 10$ \\\hline
      \SBPGamma & 4 & 18 & 7 & $5\times 5$ \rule{0ex}{3ex} \\
      \SBPOmega & 4 & 15 & 7 & $5\times 15$ \\\hline
    \end{tabular}
  \end{center}
\end{table}
}
%
%
\section{Numerical verifications}\label{sec:results}

In this section, we use numerical experiments to demonstrate the accuracy,
conservation, and stability properties of multi-dimensional SBP-SAT
discretizations.  These experiments are intended to verify the theory developed
in Sections~\ref{sec:var} and~\ref{sec:concrete}.  Before presenting the
individual verifications, we first describe their common features.

Each experiment is based on the linear advection PDE with a divergence-free
velocity field, Equation~\eqref{eq:divfree}.  In all cases the domain is the
unit square, $\Omega = [0,1]^2$, and the boundary conditions are periodic:
$\fnc{U}(0,y,t) = \fnc{U}(1,y,t)$ and $\fnc{U}(x,0,t) = \fnc{U}(x,1,t)$.

For each SBP element, we introduce a curvilinear coordinate
transformation\linebreak $(x(\xi,\eta),y(\xi,\eta))$.  Under this
transformation, it is straightforward to show that \eqref{eq:divfree} is
equivalent to
\begin{equation}
  \frac{\partial\fnc{J} \fnc{U}}{\partial t} + \frac{1}{2}\vecnabla_{\xi} \cdot \left( \veclambda_{\xi} \fnc{U}\right) + \frac{1}{2} \veclambda_{\xi} \cdot \vecnabla_{\xi} \left( \fnc{U}\right) = 0, \quad \text{where}\qquad \vecnabla_{\xi}\cdot \veclambda_{\xi} = 0.\label{eq:curvyadvect}
\end{equation}
Thus, the transformed velocity field is divergence-free in the space
$(\xi,\eta)$.

We consider a monolithic coordinate transformation that is applied over the
entire $(x,y)$ domain, because this simplifies mesh refinement studies by
permitting uniform grid refinement in $(\xi,\eta)$ space.  Let $N$ denote the
number of element edges along the $\xi$ and $\eta$ coordinates.  The vertices of
the elements are located at $(\xi_{i},\eta_{j}) = (ih, jh), \; \forall i,j =
0,1,\ldots,N$, where $h = 1/N$.  For each of the $N^{2}$ quadrilaterals, two
right triangles are generated from the vertices
\begin{equation*}
\{(\xi_{i},\eta_{j}), (\xi_{i+1},\eta_{j}), (\xi_{i},\eta_{j+1})\}
\qquad\text{and}\qquad
\{(\xi_{i+1},\eta_{j+1}),(\xi_{i},\eta_{j+1}),(\xi_{i+1},\eta_{j})\}.
\end{equation*}
The nodes associated with these right triangles are then mapped to physical space using an analytical coordinate transformation.

The SBP-SAT spatial discretization of the PDE \eqref{eq:curvyadvect} is given by
\eqref{eq:discdivfree} with the SAT matrices defined by the upwind scheme in
Section~\ref{sec:upwind}.  As explained in Section \ref{sec:divfree}, the
discretization \eqref{eq:discdivfree} must satisfy \eqref{eq:condc_upwind} to
achieve discrete conservation.  To this end, we project the analytical advection
field onto a discrete field that satisfies \eqref{eq:divfree_h}, the discrete
divergence-free equation.  The details of this projection can be found in
Appendix \ref{sec:Dlambda}.

The SBP-SAT semi-discretizations are advanced in time using the classical
4th-order Runge-Kutta scheme with a sufficiently small time step to ensure that
the error is dominated by the spatial discretization.  In particular, the time
step is one half the maximally stable value permitted by the Courant number for
a given SBP element, where the Courant number is defined as
\begin{equation*}
\mathsf{CFL} = \frac{\Delta t \|\veclambda_{\xi} \|}{h \Delta r},
\end{equation*}
for a time step of $\Delta t$ and a nominal node spacing of $h\Delta r$.  Here,
$\Delta r$ is the minimum distance between cubature nodes on a right triangle
with vertices at $(0,0)$, $(1,0)$ and $(0,1)$.  Table~\ref{tab:courant} lists
$\Delta r$ and the maximally stable Courant numbers for the \SBPGamma and
\SBPOmega elements when applied to constant-coefficient advection with
$\lamxx=\lamyy$.

\begin{table}[tbp]
  \caption[]{Maximally stable Courant numbers and minimum node spacing for
    discretizations of constant-coefficient advection based on the \SBPGamma and
    \SBPOmega operators.\label{tab:courant}}
  \begin{center}
    \begin{tabular}{llllll}
      & & \textbf{p=1} & \textbf{p=2} & \textbf{p=3} & \textbf{p=4} \\\hline
      \rule{0ex}{3ex}\SBPGamma & $\mathsf{CFL}_{\max}$ & 0.7500 & 1.3398 &
      1.2045 & 1.1597 \\
        & $\Delta r$ & 1.0000 & 0.2357 & 0.1487 & 0.0949 \\\hline
      \rule{0ex}{3ex}\SBPOmega & $\mathsf{CFL}_{\max}$ & 0.5217 & 0.4130 &
      0.3083 & 0.3428 \\
        & $\Delta r$ & 0.5000 & 0.3378 & 0.2402 & 0.1636 \\
    \end{tabular}
  \end{center}
\end{table}

\subsection{Constant-coefficient advection with a curvilinear coordinate mapping}

As our first verification of the SBP-SAT discretizations, we conduct a mesh
refinement study and discretize the constant-coefficient advection equation with
$\veclambda = \left[1,1\right]\Tr$.  While this PDE does not have a spatially
varying velocity field, we employ a curvilinear coordinate transformation given
by
\begin{equation*}
\begin{bmatrix} x \\ y \end{bmatrix}
=
\begin{bmatrix}
  \xi + \frac{1}{5} \sin(\pi\xi)\sin(\pi\eta) \\
  \eta - \frac{1}{5}\exp(\eta)\sin(\pi\xi)\sin(\pi\eta)
\end{bmatrix},
\end{equation*}
where $(\xi,\eta) \in [0,1]^{2}$.  Consequently, the transformed PDE,
~\eqref{eq:curvyadvect}, does have a spatially-varying velocity field even
though the physical-space PDE does not.  The sequence of grids for the mesh refinement
study is generated as described earlier using $N \in \{12, 24, 36, 48, 60, 72\}$.
The initial condition for the accuracy study is a bell-shaped function centered
at $\left(\frac{1}{2},\frac{1}{2}\right)$ with compact support:
\begin{equation*}
\fnc{U}(x,y,0) = \begin{cases}
    1 - (4\rho^2-1)^5 & \text{if}\; \rho \leq \frac{1}{2} \\
    1, &\text{otherwise},
  \end{cases}
\end{equation*}
where $\rho(x,y) \equiv \sqrt{(x-\frac{1}{2})^2 + (y-\frac{1}{2})^2}$.  The
solution is advanced one time unit, which returns the bell-shaped initial
condition to its initial position.

\subsubsection{Accuracy}

To assess the accuracy of the discrete solutions, we evaluate the SBP-based
$L^{2}$ norm of the difference between the numerical solution and the exact
solution.  We then normalize by the norm of the exact solution; that is,
\begin{equation*}
\textsf{Normalized}\; L^{2}\; \textsf{Error} =
\frac{\sqrt{(\bm{u} - \bm{u}_{e})\Tr \Mg (\bm{u} - \bm{u}_{e})}}{\sqrt{\bm{u}_{e}\Tr\M_{g}\bm{u}_{e}}},
\end{equation*}
where $\bm{u}$ is the discrete solution at the final time, and $\bm{u}_{e}$ is
the exact solution evaluated at the mesh nodes at time $t=1$.  The matrix $\Mg$
is the global SBP-norm assembled from the local element SBP-norm matrices scaled
by the appropriate mapping Jacobian determinant on each element, \ie it is the
diagonal mass matrix.

The accuracy results of the mesh refinement study are shown in
Figure~\ref{fig:accuracy_curvy} for discretizations based on the \SBPGamma and
\SBPOmega families of operators.  The expected asymptotic convergence rate for
the errors is $\text{O}(h^{p+1})$, and most of the schemes exhibit this
convergence rate.  The \SBPGamma $p=1$ discretization is the only scheme that
has a suboptimal convergence rate for the range of meshes considered --- it could be the case that the sequence of meshes was not sufficiently fine so as to be in the asymptotic region.  The
\SBPGamma $p=2$ scheme and the \SBPOmega $p=1$ and $p=2$ schemes exhibit
$\text{O}(h^{p+2})$ rates.

For the same $h$ and $p$, the two SBP families produce notably different
absolute errors.  The difference is especially significant for the $p=1$ and
$p=2$ schemes.  On the finest grid, the error in the \SBPOmega $p=1$ solution is
16.6 times smaller than the corresponding error in the \SBPGamma $p=1$ solution.
The solution errors of the $p=2$ schemes differ by a factor of 5 on the finest
grid.  We believe this difference is related to the increased accuracy of the
SBP cubature rules associated with the \SBPOmega schemes.

\begin{figure}[tbp]
  \subfigure[\SBPGamma family \label{fig:accuracy_SBPGamma}]{%
    \includegraphics[width=0.48\textwidth]{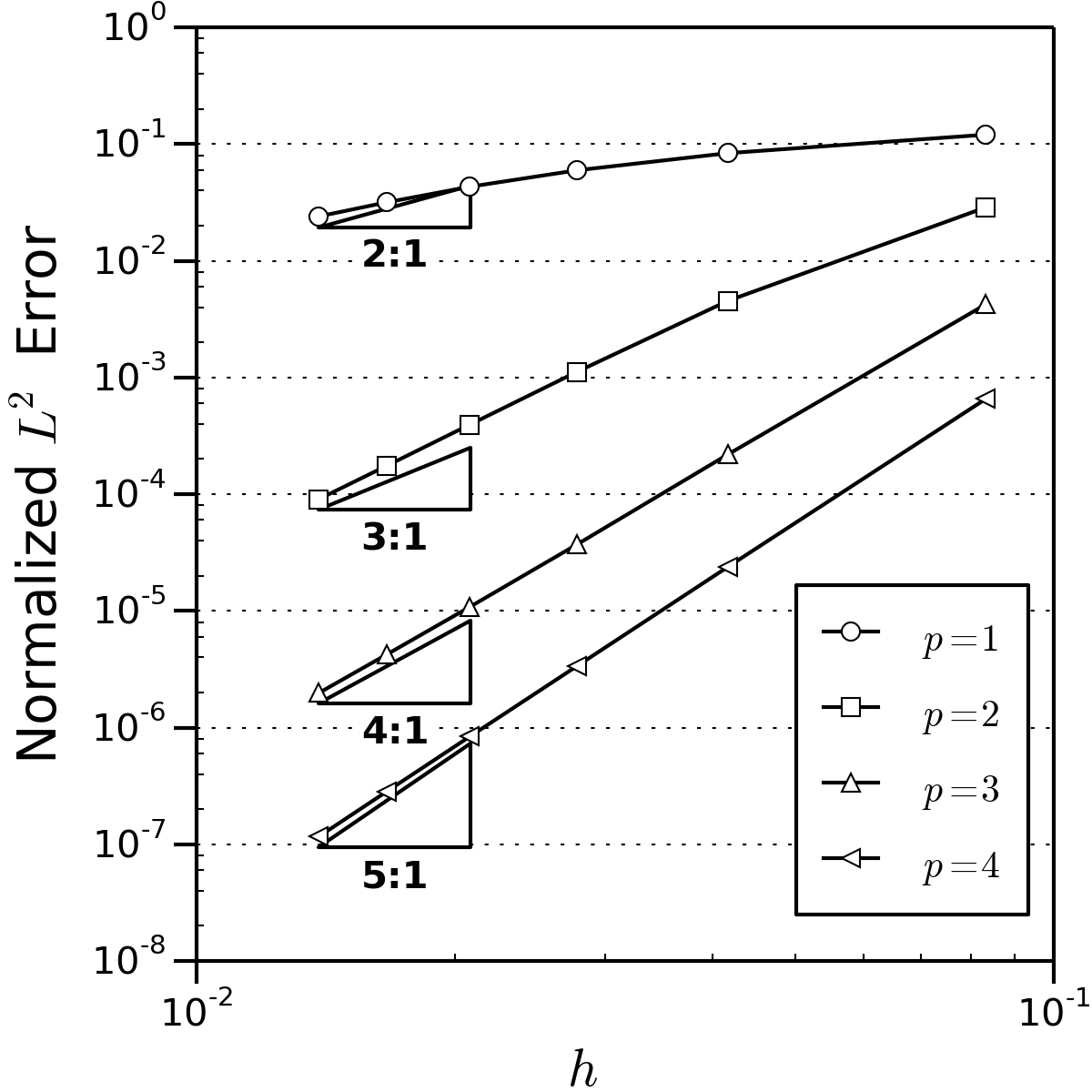}}
  \subfigure[\SBPOmega family \label{fig:accuracy_SBPOmega}]{%
    \includegraphics[width=0.48\textwidth]{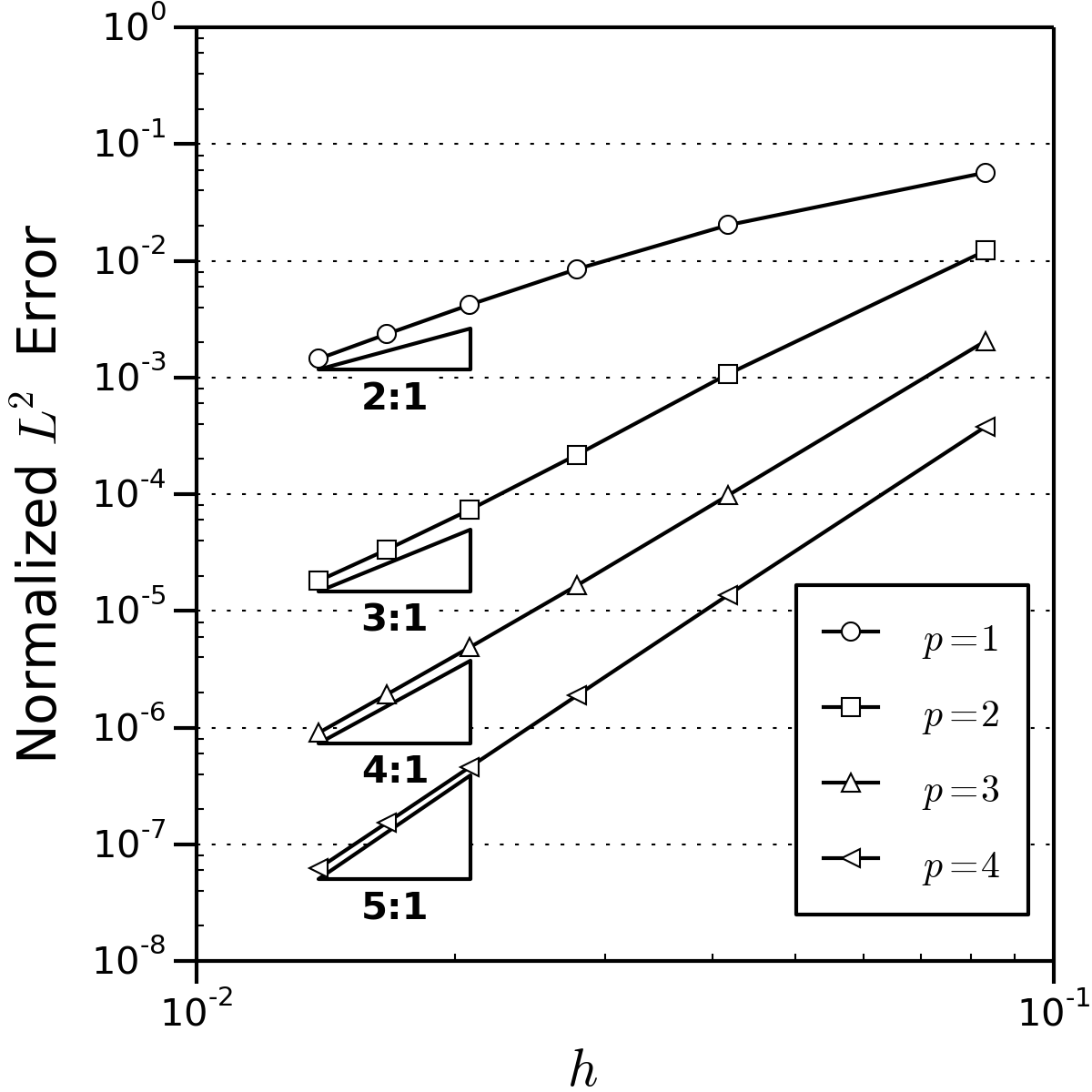}}
  \caption{Normalized error, measured in the SBP-norm, between the discrete and exact
    solutions to \eqref{eq:curvyadvect} for different mesh spacing and SBP
    operators.\label{fig:accuracy_curvy}}
\end{figure}

\subsubsection{Conservation}

The integral of the analytical solution of \eqref{eq:curvyadvect} is constant
in time, because the PDE is conservative and the boundary conditions are
periodic.  Based on the analysis in Section~\ref{sec:var}, the SBP-SAT
discretization should mimic this property, and the schemes should be
conservative to machine precision.

Discrete conservation is assessed using the following metric:
\begin{equation*}
  \mathsf{Conservation\; Metric} \equiv |\bm{1}^{T} \Mg \bm{u}_{0} - \bm{1}^{T}
  \Mg \bm{u}|,
\end{equation*}
where $\bm{u}_{0}$ is the initial condition evaluated at the nodes and, as
before, $\bm{u}$ is the discrete solution at $t=1$.
Figure~\ref{fig:conserve_and_stab} plots this metric for the \SBPGamma and
\SBPOmega discretizations on each of the grids in the mesh refinement study.
These results provide strong evidence that the SBP-SAT discretizations are
conservative.  Note that the \SBPOmega $p=2$ scheme produces a double-precision
zero for the conservation metric on the coarsest grid, which cannot be
represented on the logarithmic scale.

\subsubsection{Stability}

The $L^2$ norm of the analytical solution to \eqref{eq:curvyadvect} is also
constant in time; however, in contrast with conservation, the energy of the
SBP-SAT discrete solution is only guaranteed to be non-increasing when upwind
SATs are used, in general.  To assess the various schemes' ability to conserve
energy, we evaluate the energy error for each mesh and operator:
\begin{equation*}
  \mathsf{Energy\; Error} \equiv \bm{u}_{0}^{T} \Mg \bm{u}_{0} - \bm{u}^{T} \Mg
  \bm{u}.
\end{equation*}

The energy errors are included in Figure~\ref{fig:conserve_and_stab} above the
conservation metrics.  Since the energy error is the signed difference between
the initial and final values, it offers some evidence that the energy is
non-increasing; stronger evidence is provided below in Section~\ref{sec:robust}

\begin{remark}
The rate of convergence of the energy error is approximately $2p$ for the
\SBPGamma schemes and $2p+1$ for the \SBPOmega schemes.  This is an example of
functional superconvergence, which has also been observed and explained for
tensor-product SBP schemes~\cite{Hicken2011superconvergent}.
\end{remark}

\begin{figure}[tbp]
  family \label{fig:conserve_and_stab_gamma}]{%
      \includegraphics[width=0.48\textwidth]{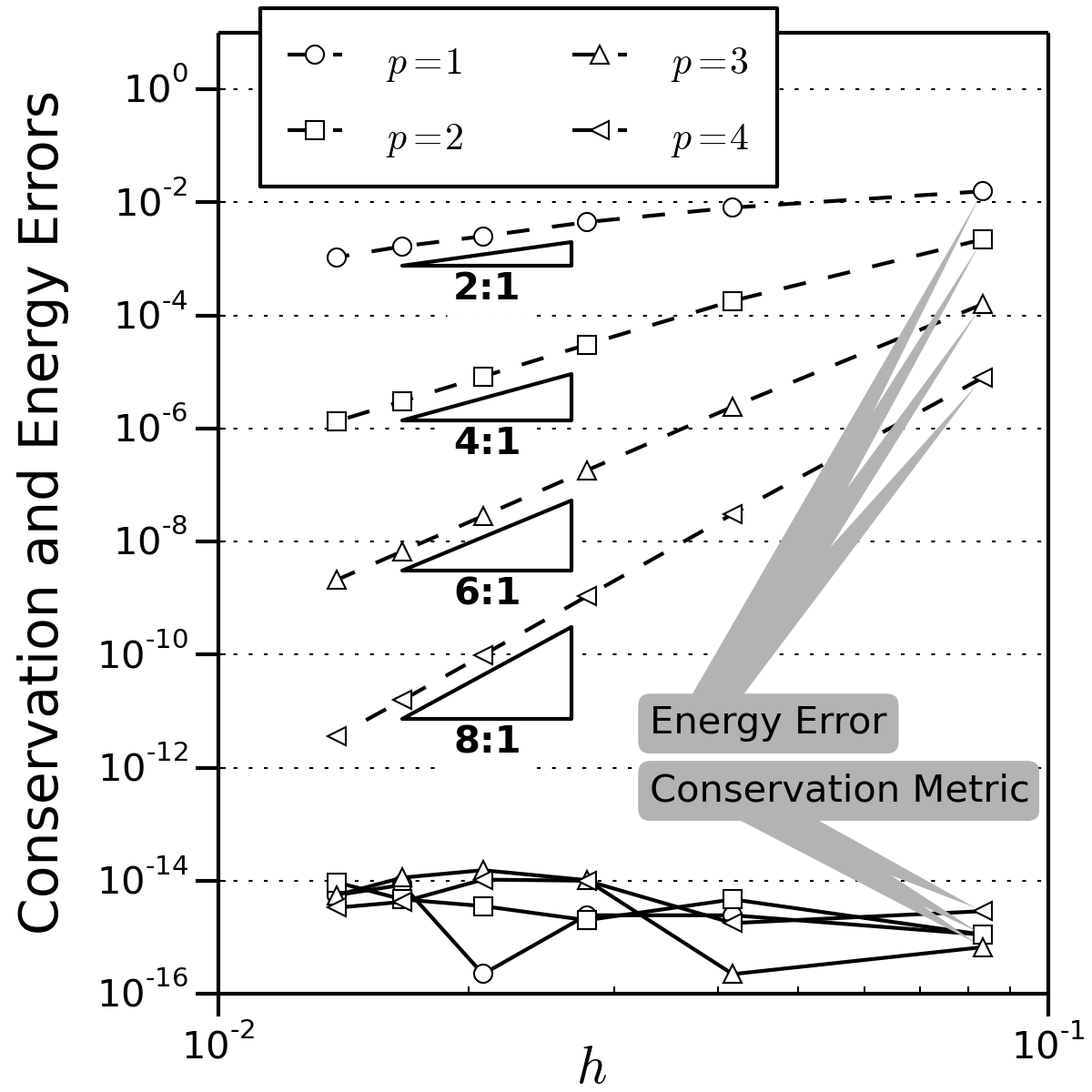}}
    \subfigure[\SBPOmega family \label{fig:conserve_and_stab_omega}]{%
      \includegraphics[width=0.48\textwidth]{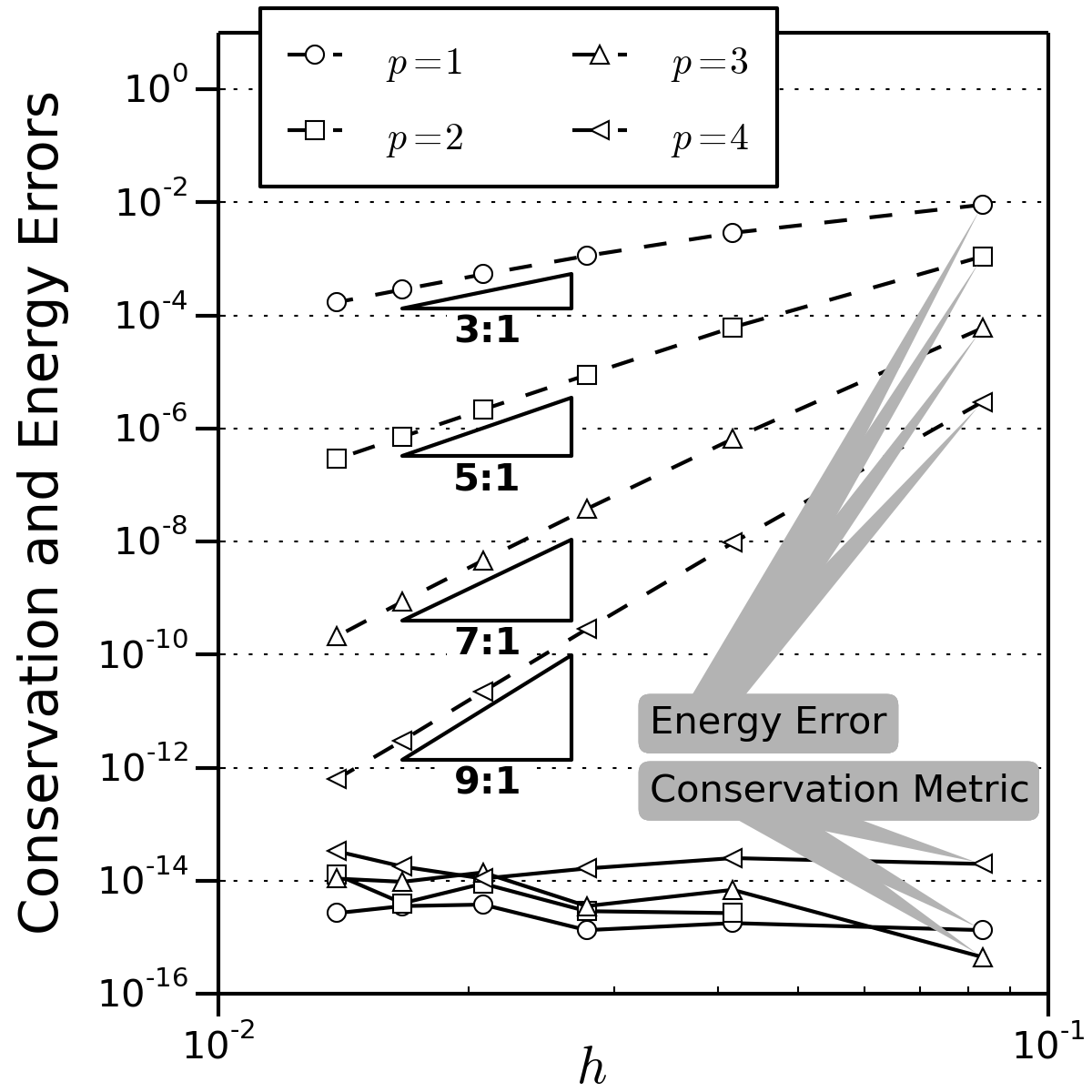}}
  \caption{Conservation and energy errors for different mesh spacing and SBP
    operators.\label{fig:conserve_and_stab}}
\end{figure}

\subsection{Robustness of SBP-SAT discretizations: advection in a confined domain}
\label{sec:robust}

For the second numerical experiment, we consider a challenging test of
the numerical stability of the SBP-SAT discretizations.  The test case is
challenging, because the advection field,
\begin{equation*}
  \veclambda = \begin{bmatrix}
    \pi \sin(\pi x) \cos(\pi y), \\
  -\  \pi \cos(\pi x) \sin(\pi y)
  \end{bmatrix},
\end{equation*}
is parallel to the boundary of the domain and produces no boundary flux; thus,
the solution, and its energy, are confined to the domain.  In addition, the
nonpolynomial velocity and solution produce aliasing errors that the numerical
scheme must handle ``gracefully.''

The initial condition is given by $\fnc{U}(x,y,0) = \exp(xy)$, and the solution
is advanced for 10 nondimensional time units on a uniform grid with $N=12$ edges
in each direction, \ie there are $2N^2 = 288$ elements in total.  As before, the
time step is set such that the Courant number is one half the value of
$\mathsf{CFL}_{\max}$ listed in Table~\ref{tab:courant}; however, we emphasize
that we are only interested in assessing the stability of the methods with this
experiment, and the discrete solutions after 10 time units are not accurate for
the coarse grids considered.  To give some indication of the solution behavior
and the time duration, Figure~\ref{fig:robust_solution} shows the initial
solution and the exact solution after \emph{only one unit of time}.

\begin{figure}[tp]
  \subfigure[initial condition \label{fig:divfree_initial}]{%
    \includegraphics[width=0.49\textwidth]{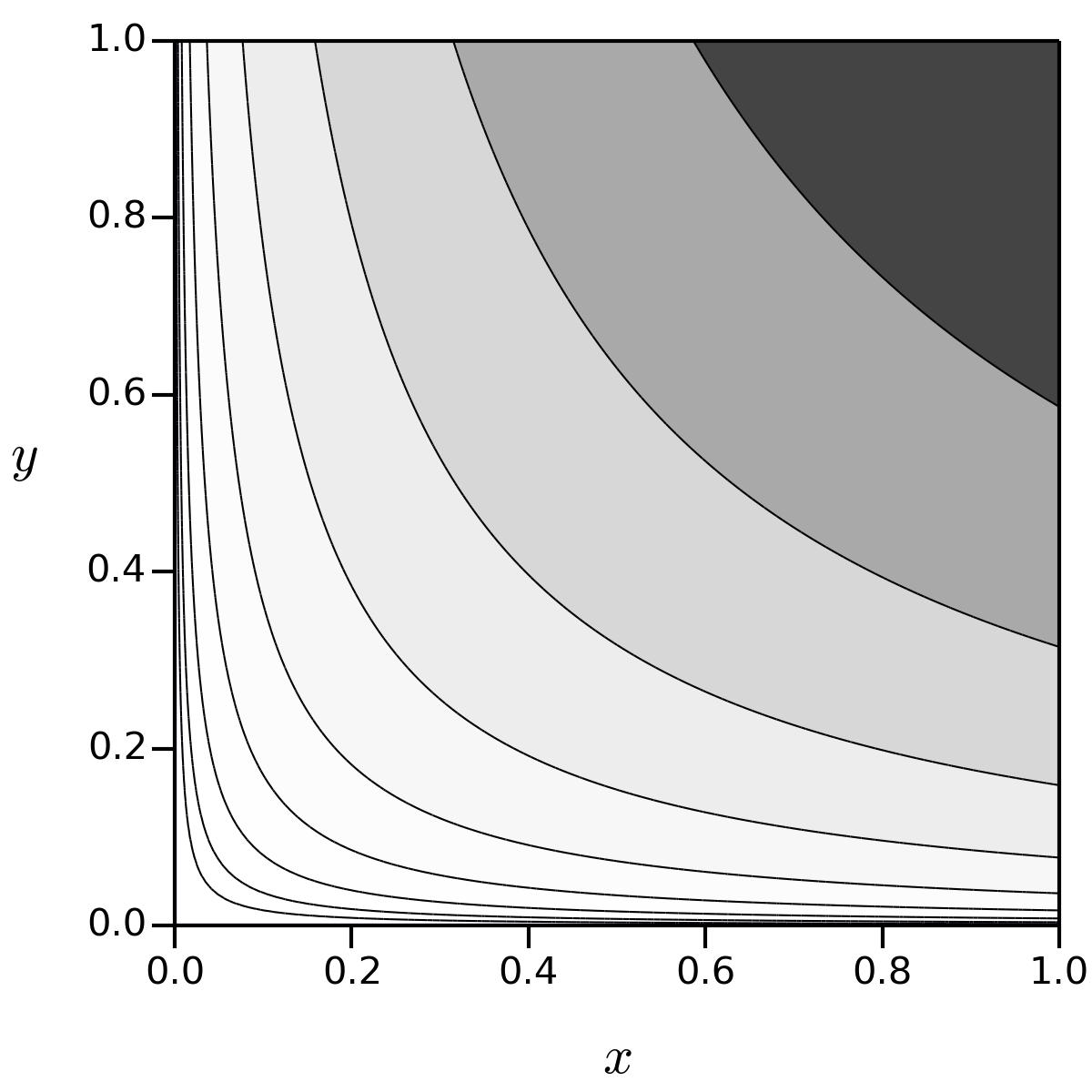}}
  \subfigure[exact solution at $t=1$ \label{fig:divfree_exact_solution}]{%
    \includegraphics[width=0.49\textwidth]{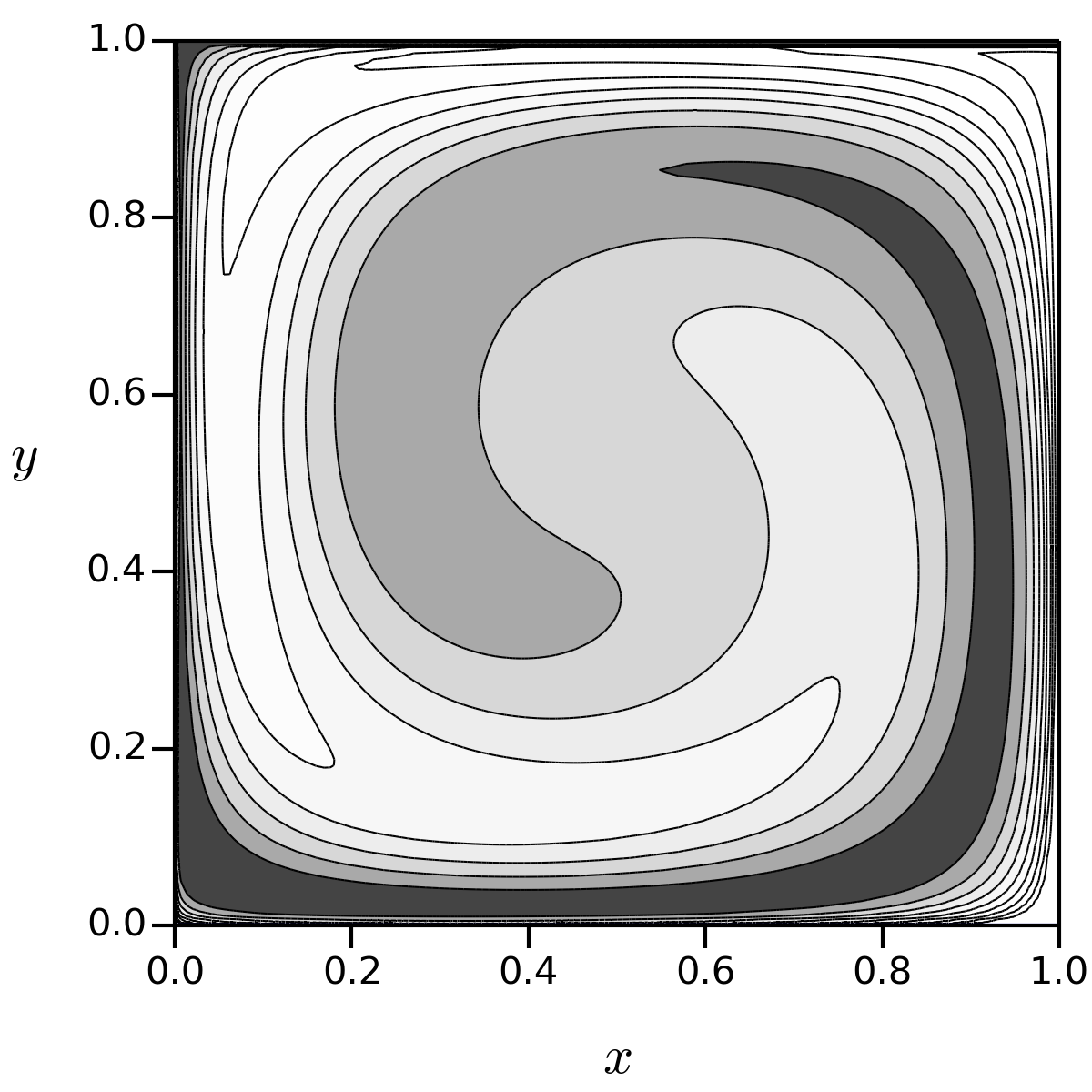}}
  \caption{Initial condition, left, and exact solution at $t=1$ for robustness test.\label{fig:robust_solution}}
\end{figure}

To demonstrate that the SBP-SAT discretizations are energy stable,
Figure~\ref{fig:energy_hist} shows the change in normalized energy as the
discrete solutions evolve from $t=0$ to $t=10$.  The normalized change in energy
is given by
\begin{equation*}
  \frac{\bm{u}(t)^{T} \Mg \bm{u}(t) - \bm{u}_{0}^{T} \Mg \bm{u}_{0}}{%
    \bm{u}_{0}^{T} \Mg \bm{u}_{0}} = \frac{\| \bm{u}(t)\|_{\Mg}^{2}}{%
      \| \bm{u}_{0} \|_{\Mg}^{2}} - 1
\end{equation*}
where $\bm{u}(t)$ denotes the discrete solution at time $t$.  As with the
conservation metric, we consider a uniform triangulation with $N=12$ edges in
each direction and 288 elements total.

Figures~\ref{fig:energy_hist_bndry} and \ref{fig:energy_hist_intr} show the
change in energy for the \SBPGamma and \SBPOmega families applied to the
skew-symmetric discretization~\eqref{eq:discdivfree} with upwind SATs.  The plots show
that the SBP-SAT discretizations have nonincreasing energies, as expected from
the analysis in Sections~\ref{sec:var} and \ref{sec:upwind}.  In contrast,
Figure~\ref{fig:energy_hist_bndryNDG} shows the change in energy for the
\SBPGamma family applied to the ``divergence'' form of the discretization,
namely
\begin{equation*}
\frac{d\mat{J} \bm{u}}{dt} + \Dx\Lamx \bm{u} + \Dy\Lamy \bm{u} = \bm{\mr{SAT}}_{\bm{u}}.
\end{equation*}
As the plots show, only the skew-symmetric discretizations have bounded
energies\footnote{For this problem the discretization of the divergence form
  leads to increasing energy, but this is not always the case. Indeed, when
  solving the constant-coefficient, curvilinear-coordinate problem we did not
  encounter increasing energy.}

\begin{figure}[tbp]
  \subfigure[\SBPGamma family \label{fig:energy_hist_bndry}]{%
    \includegraphics[width=0.9\textwidth]{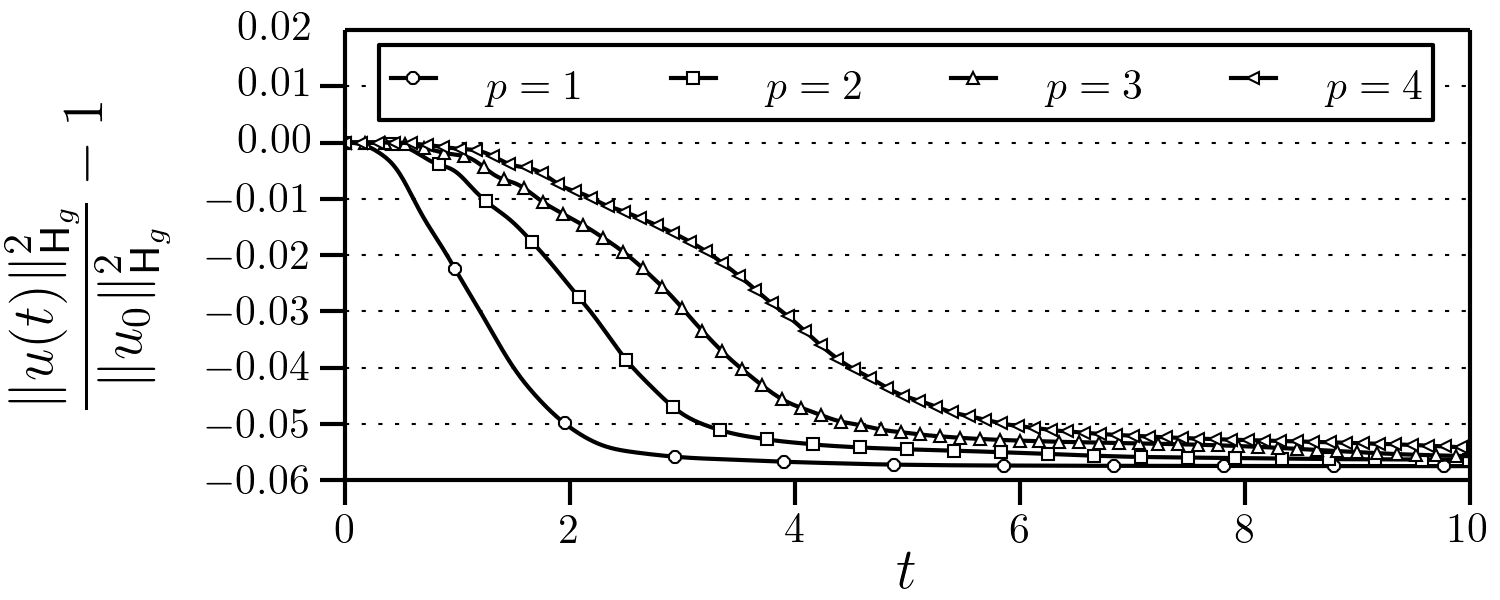}}\\
  \subfigure[\SBPOmega family \label{fig:energy_hist_intr}]{%
    \includegraphics[width=0.9\textwidth]{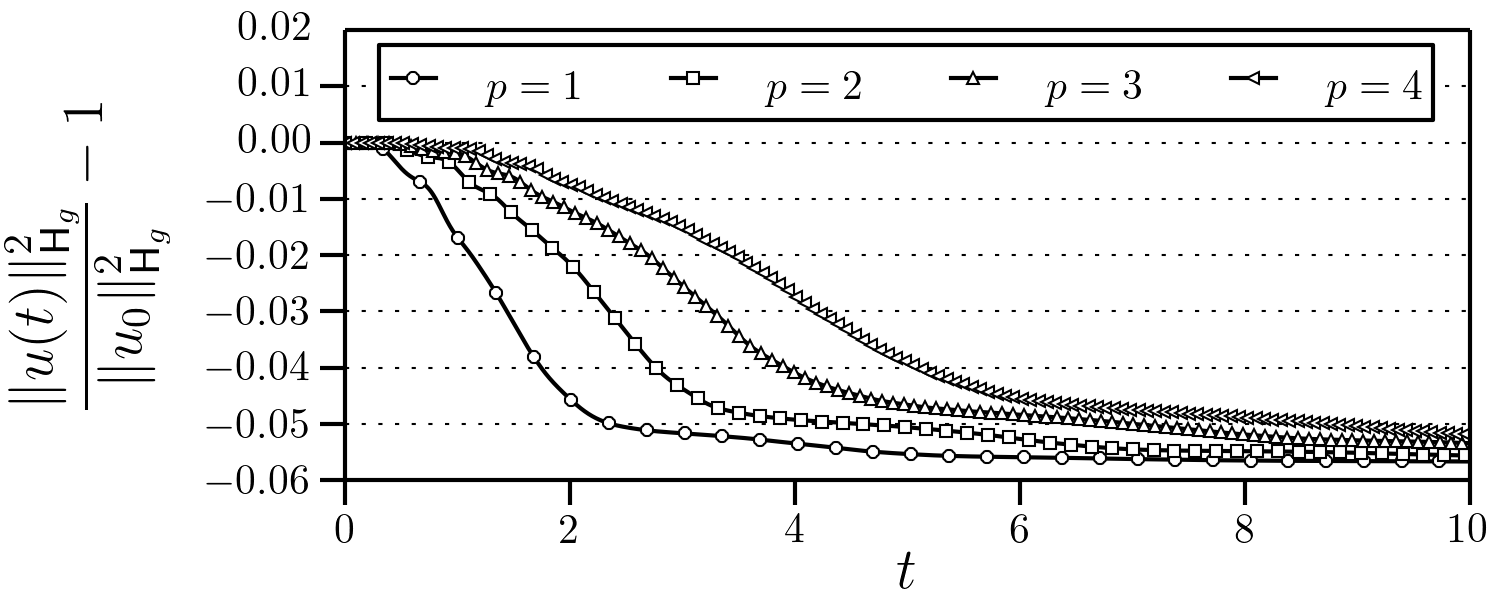}}\
  \subfigure[\SBPGamma with divergence formulation \label{fig:energy_hist_bndryNDG}]{%
    \includegraphics[width=0.9\textwidth]{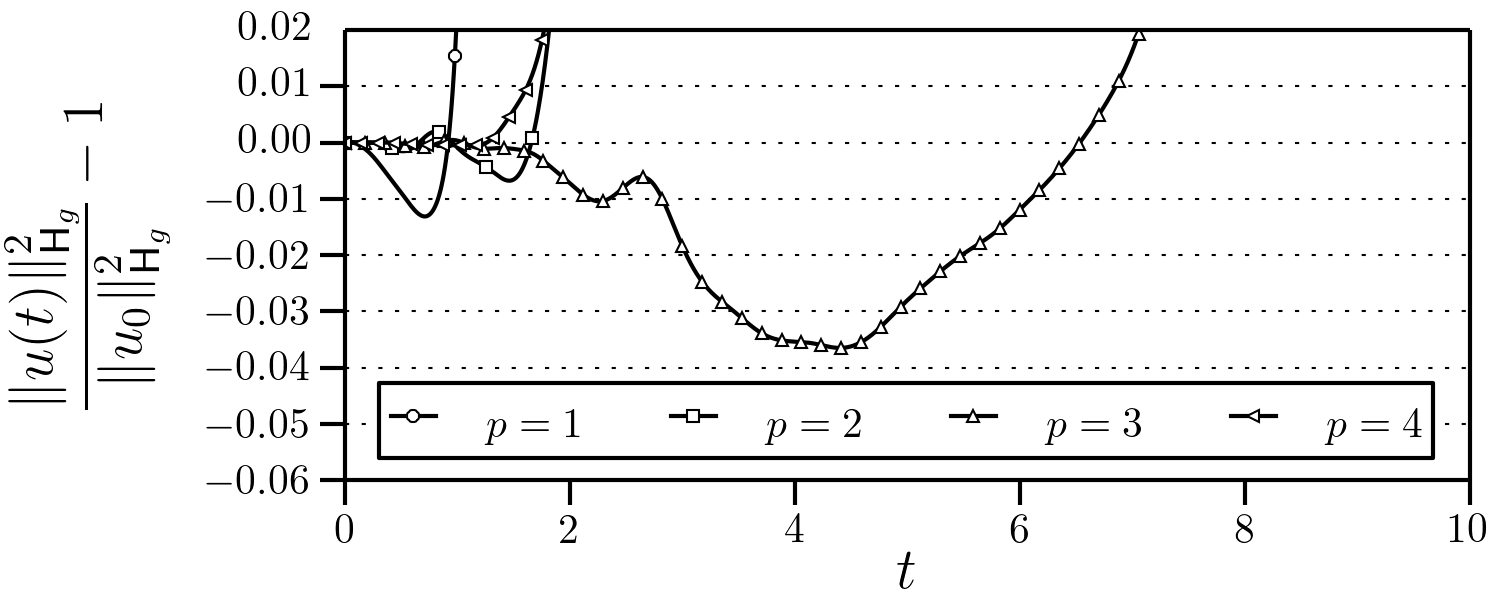}}\\
  \caption{Normalized change in energy versus time.  Every 100th time-step is
    marked with a symbol.\label{fig:energy_hist}}
\end{figure}


%
%
\section{Conclusions}\label{sec:conclude}

Multi-dimensional SBP operators offer time-stable, high-order, and conservative
discretizations on complex domains, but only if boundary conditions and
inter-element coupling can be imposed in a suitable manner.  To this end, we
have proposed a general framework for the development of SATs that lead to accurate,
stable, and conservative schemes. We focused on developing a set of SATs that are
simple to construct and that allow for the pointwise imposition of boundary
conditions and inter-element coupling. This was accomplished by using
interpolation/extrapolation operators and face-based cubatures to construct the
coupling terms in the SATs. A key insight of this paper is that the $\mat{E}$ matrices and
the coupling terms in the SATs can be decomposed in the same way;
this insight significantly simplifies the development of this class of SATs.

Using these SATs, we showed how to derive conservative and time-stable
discretizations for multi-dimensional SBP operators in the context of the linear
advection equation with a spatially varying velocity field.  In this context,
conservation requires a particular relationship between the
interpolated/extrapolated fluxes and the SATs.  For a divergence-free problem,
we satisfied this conservation condition by projecting the analytical advection
field onto a field that satisfies a discrete form of the divergence-free
equation.  For nonlinear hyperbolic systems of PDEs, numerical flux functions
can be used to satisfy the conservation condition.

The SAT methodology was illustrated using SBP operators on triangular
elements. Two SBP families were considered: the \SBPGamma family with $p+1$
nodes on each face and the \SBPOmega family with strictly interior nodes.

The accuracy, conservation, and stability properties of the SBP-SAT
discretizations were verified using the linear advection equation with
divergence-free velocity fields.  Both the \SBPOmega and \SBPGamma schemes were
shown to be conservative to machine precision, and both produced non-increasing
energy.  For the same operator degree $p$, the \SBPOmega scheme was found to be
more accurate. Finally, we numerically demonstrated that the SBP-SAT
discretizations presented result in superconvergent functional estimates.

\appendix
%
%
\normalsize
\section{Satisfaction of the discrete divergence-free equation}\label{sec:Dlambda}

In general, the analytical velocity, which we will denote here as $\hat{\veclambda}_{\xi}$, does not satisfy the
discretized divergence-free condition, \eqref{eq:divfree_h}.  Therefore, we seek
a discrete vector field that satisfies the discrete divergence-free condition
and is as close as possible, in some norm, to the analytical field.  This
appendix describes how find such a discrete vector field.

First we solve for the face-normal velocities, $(\lambda_n)_{i} = \left(\lambda_{\xi} n_{\xi} + \lambda_{\eta} n_{\eta}\right)_{i}$, that appear in
the elements of the $\Blam$ matrices.  A constraint on the $(\lambda_n)_{i}$ for
each element is obtained by substituting $\bm{v} = \bm{1}$ into the
identity~\eqref{eq:divfree_identity}:
\begin{equation*}
  \sum_{j=1}^{\kappa} \left(\bm{1}_{\hat{\Gamma}_{j}}\Tr\B_{\lambda,j} \R_{j}\right)\bm{1} = \sum_{j=1}^{\kappa} \sum_{i=1}^{n_{j}} b_{i}^{(j)} \left(\lambda_{n}^{(j)}\right)_{i}
 = \bm{1}^{T}\left(\Lamx \Qx + \Lamy \Qy\right) \bm{1} = 0,
\end{equation*}
where the last equality follows from $\Qx \bm{1} = \Qy \bm{1} = \bm{0}$.  This
constraint is simply a discretization of $\int_{\Gamhat} \veclambda \cdot
\vecnrm\, \mr{d}\Gamhat = 0$ on each element.  There are fewer elements than
face-normal velocities, so we solve a quadratic optimization problem that
minimizes the Cartesian norm between the discrete and analytical values at the face nodes, $(\lambda_{n}^{(j)})_{i}$ and $(\hat{\lambda}_{n}^{(j)})_{i}$, respectively, subject to the above constraint.

Once the $\Blam$ matrices are determined, we solve for the diagonal matrices
$\Lambda_{\xi}$ and $\Lambda_{\eta}$.  We follow a procedure analogous to the
one used for $\Blam$; in this case we minimize the Cartesian norm between the discrete and analytical values at the SBP nodes and \eqref{eq:divfree_h} becomes the constraint.  The optimization problems on each element are decoupled.

For the cases considered here, we verified that the $L^2$ error in the discrete
velocity field is at least an order of magnitude smaller than the $L^2$ error in
the scalar field $\fnc{U}$.  Moreover, the error in the velocity field decreases
with $h$, the average mesh spacing, at a faster rate than the error in $\bm{u}$,
and the error in the velocity field has an insignificant impact on the solution
error.

\bibliographystyle{aiaa}
\bibliography{references}

\end{document}